\documentclass[reqno]{amsart}
\usepackage{amsmath, amssymb, amsthm, epsfig}
\usepackage{hyperref, latexsym}
\usepackage{url} 
\usepackage{cellspace}
\DeclareMathOperator{\sgn}{\mathrm{sgn}}

\begin{document}
 \bibliographystyle{plain}

 \newtheorem{theorem}{Theorem}
 \newtheorem{lemma}[theorem]{Lemma}
 \newtheorem{proposition}[theorem]{Proposition}
 \newtheorem{corollary}[theorem]{Corollary}
 \theoremstyle{definition}
 \newtheorem{definition}[theorem]{Definition}
 \newtheorem{example}[theorem]{Example}
 \theoremstyle{remark}
 \newtheorem{remark}[theorem]{Remark}
 \newcommand{\mc}{\mathcal}
 \newcommand{\A}{\mc{A}}
 \newcommand{\B}{\mc{B}}
 \newcommand{\cc}{\mc{C}}
 \newcommand{\D}{\mc{D}}
 \newcommand{\E}{\mc{E}}
 \newcommand{\F}{\mc{F}}
 \newcommand{\G}{\mc{G}}
 \newcommand{\sH}{\mc{H}}
 \newcommand{\I}{\mc{I}}
 \newcommand{\J}{\mc{J}}
 \newcommand{\K}{\mc{K}}
 \newcommand{\lL}{\mc{L}}
 \newcommand{\M}{\mc{M}}
 \newcommand{\nn}{\mc{N}}
 \newcommand{\rr}{\mc{R}}
 \newcommand{\sS}{\mc{S}}
 \newcommand{\U}{\mc{U}}
 \newcommand{\X}{\mc{X}}
 \newcommand{\Y}{\mc{Y}}
 \newcommand{\C}{\mathbb{C}}
 \newcommand{\R}{\mathbb{R}}
 \newcommand{\N}{\mathbb{N}}
 \newcommand{\Q}{\mathbb{Q}}
 \newcommand{\Z}{\mathbb{Z}}
 \newcommand{\csch}{\mathrm{csch}}
 \newcommand{\tF}{\widehat{F}}
 \newcommand{\tG}{\widehat{G}}
 \newcommand{\tH}{\widehat{H}}
 \newcommand{\tf}{\widehat{f}}
 \newcommand{\ug}{\widehat{g}}
 \newcommand{\wg}{\widetilde{g}}
 \newcommand{\uh}{\widehat{h}}
 \newcommand{\wh}{\widetilde{h}}
 \newcommand{\wl}{\widetilde{l}}
 \newcommand{\tk}{\widehat{k}}
 \newcommand{\tK}{\widehat{K}}
 \newcommand{\tl}{\widehat{l}}
 \newcommand{\tL}{\widehat{L}}
 \newcommand{\tm}{\widehat{m}}
 \newcommand{\tM}{\widehat{M}}
 \newcommand{\tp}{\widehat{\varphi}}
 \newcommand{\tq}{\widehat{q}}
 \newcommand{\tT}{\widehat{T}}
 \newcommand{\tU}{\widehat{U}}
 \newcommand{\tu}{\widehat{u}}
 \newcommand{\tV}{\widehat{V}}
 \newcommand{\tv}{\widehat{v}}
 \newcommand{\tW}{\widehat{W}}
 \newcommand{\ba}{\boldsymbol{a}}
 \newcommand{\bal}{\boldsymbol{\alpha}}
 \newcommand{\bx}{\boldsymbol{x}}
 \newcommand{\p}{\varphi}
 \newcommand{\f}{\frac52}
 \newcommand{\g}{\frac32}
 \newcommand{\h}{\frac12}
 \newcommand{\hh}{\tfrac12}
 \newcommand{\ds}{\text{\rm d}s}
 \newcommand{\dt}{\text{\rm d}t}
 \newcommand{\du}{\text{\rm d}u}
 \newcommand{\dv}{\text{\rm d}v}
 \newcommand{\dw}{\text{\rm d}w}
 \newcommand{\dx}{\text{\rm d}x}
 \newcommand{\dy}{\text{\rm d}y}
 \newcommand{\dl}{\text{\rm d}\lambda}
 \newcommand{\dmu}{\text{\rm d}\mu(\lambda)}
 \newcommand{\dnu}{\text{\rm d}\nu(\lambda)}
\newcommand{\dnus}{\text{\rm d}\nu_{\sigma}(\lambda)}
 \newcommand{\dlnu}{\text{\rm d}\nu_l(\lambda)}
 \newcommand{\dnnu}{\text{\rm d}\nu_n(\lambda)}
\newcommand{\sech}{\text{\rm sech}}
 \def\today{\number\time, \ifcase\month\or
  January\or February\or March\or April\or May\or June\or
  July\or August\or September\or October\or November\or December\fi
  \space\number\day, \number\year}

\title[The Beurling-Selberg extremal problem]{Gaussian Subordination for the Beurling-Selberg Extremal Problem}
\author[Carneiro, Littmann and Vaaler]{Emanuel Carneiro, Friedrich Littmann and Jeffrey D. Vaaler}

\date{\today}
\subjclass[2000]{Primary 41A30, 41A52. Secondary 41A05, 41A44, 42A82}
\keywords{Gaussian, exponential type, extremal functions, majorization, tempered distributions.}
\address{School of Mathematics, Institute for Advanced Study, Princeton, NJ 08540.}
\email{ecarneiro@math.ias.edu}
\address{Department of mathematics, North Dakota State University, Fargo, ND 58105-5075.}
\email{friedrich.littmann@ndsu.edu}
\address{Department of Mathematics, University of Texas at Austin, Austin, TX 78712-1082.}
\email{vaaler@math.utexas.edu}

\begin{abstract} We determine extremal entire functions for the problem of majorizing, minorizing, and approximating the Gaussian function $e^{-\pi\lambda x^2}$ by entire functions of exponential type. The combination of the Gaussian and a general distribution approach provides the solution of the extremal problem for a wide class of even functions that includes most of the previously known examples (for instance \cite{CV2}, \cite{CV3}, \cite{GV} and \cite{Lit}), plus a variety of new interesting functions such as $|x|^{\alpha}$ for $-1 < \alpha$; \,$\log \,\bigl((x^2 + \alpha^2)/(x^2 + \beta^2)\bigr)$, for $0 \leq \alpha < \beta$;\, $\log\bigl(x^2 + \alpha^2\bigr)$; and $x^{2n} \log x^2$\,, for $n \in \N$. Further applications to number theory include optimal approximations of theta functions by trigonometric polynomials and optimal bounds for certain Hilbert-type inequalities related to the discrete Hardy-Littlewood-Sobolev inequality in dimension one.
\end{abstract}

\maketitle

\numberwithin{equation}{section}

\section*{Introduction}

We recall that an entire function $F:\C \to \C$ is of {\it exponential type} at most $2\pi \delta$ if for every $\epsilon >0$ there exists a
positive constant $C$, such that the inequality 
\begin{equation*}\label{intro0}
|F(z)| \leq C e^{(2\pi \delta + \epsilon) |z|}
\end{equation*}
holds for all $z \in \C$.  For a given function $f: \R \to \R$, the Beurling-Selberg extremal problem consists of finding an entire function $F(z)$ of exponential type at most $2\pi \delta$, such that the integral 
\begin{equation}\label{BS1}
 \int_{-\infty}^{\infty} |F(x) - f(x)|\, \dx
\end{equation}
is minimized.  An important variant of this problem, useful in many applications to number theory and analysis, occurs when we impose the additional condition that $F(z)$ is real valued on $\R$ and satisfies $F(x) \geq f(x)$ for all $x \in \R$.  In this case a function $F(z)$ that minimizes the integral (\ref{BS1}) is 
called an extreme majorant of $f(x)$.  Extreme minorants are defined analogously.

This extremal problem was introduced in the work of A.~Beurling in the late 1930's for the function $f(x) = \sgn(x)$.  Later A.~Selberg recognized how to majorize and minorize the characteristic function of an interval using Beurling's extremal function, and made use of this construction to obtain a sharp form of the large sieve inequality.  Further applications in analytic number theory are discussed in \cite{S1} and \cite{S2}.  An outline of the early development of this theory, including simple proofs of the Erd\"{o}s-Tur\'{a}n inequality and the Montgomery-Vaughan inequality (see \cite{MV}) is presented in \cite{V}.

General solutions to the Beurling-Selberg extremal problem for different classes of functions have been identified in several recent papers, and these have included new number theoretical applications.  A key ingredient in most applications is the close connection between entire functions of exponential type and functions with compactly supported Fourier transform via the Paley-Wiener theorem.  The extremal problem for the exponential function $f(x) = e^{-\lambda |x|}$, $\lambda>0$, is discussed by Graham and Vaaler in \cite{GV}, with applications to Tauberian theorems. The problem for $f(x) = x^n \sgn(x)$ and $f(x) = (x^+)^n$, where $n$ is a positive integer, is considered by Littmann in \cite{Lit0}, \cite{Lit} and \cite {Lit2}.  In \cite{CV2} and \cite{CV3}, Carneiro and Vaaler extended the construction of extremal approximations for a class of 
even functions that includes $f(x) = \log|x|$,  $f(x) = \log\bigl(x^2/(x^2 + 4)\bigr)$ and $f(x) = |x|^{\alpha}$, with $-1< \alpha < 1$. 

Recently, Chandee and Soundararajan in \cite{CS} used the extremal functions for $f(x) = \log\bigl(x^2/(x^2 + 4)\bigr)$ to obtain improved upper bounds for 
$|\zeta(\tfrac{1}{2} + it)|$ assuming the Riemann Hypothesis (RH). They remarked that the extremals for the function $f(x) = \log\bigl((x^2 + \alpha^2)/(x^2 + 4)\bigr)$, for $\alpha \neq 0$, not contemplated in the previous literature, naturally arise in bounding $|\zeta(\tfrac{1}{2} \pm \alpha + it)|$, assuming RH. We will return to this example later in this paper, since the family of functions $f(x) = \log \,\bigl((x^2 + \alpha^2)/(x^2 + \beta^2)\bigr)$, for $0 \leq \alpha < \beta$, is contemplated by the methods we are about to present.

Other problems on approximation by entire functions and trigonometric polynomials have been investigated by Carneiro \cite{Car}, Ganelius \cite{ganelius}, Ganzburg and Lubinsky \cite{GL}, Graham and Vaaler \cite{GV2}, Montgomery \cite{M} and Vaaler \cite{V2}.  Related extremal problems in several variables, a hard ramification of the theory, are considered by Barton, Montgomery and Vaaler \cite{BMV}, Holt and Vaaler \cite{HV} and Li and Vaaler\cite{LV}.

This paper is divided in three parts.  In the first part we consider the problem of majorizing, minorizing, and approximating the Gaussian function 
\begin{equation}\label{intro1}
x\mapsto G_{\lambda}(x) = e^{-\pi\lambda x^2}
\end{equation}
on $\R$ by entire functions of exponential type.  Here $\lambda > 0$ is a parameter.  We make use of classical interpolation techniques and integral representations to achieve this goal.

The second part is independent of the first and presents a new approach to the Beurling-Selberg extremal problem based on distribution theory. The main ingredient here is the Paley-Wiener theorem for distributions.  Our method provides the solution of the extremal problem for a wide class of even functions once one knows the classical solution for a family of even functions with an independent parameter.  In the present work this family of even
function is given by (\ref{intro1}) for $\lambda > 0$.  

Various applications are presented in the third part of the paper.
Most of the previously known cases become corollaries of this method (for instance, the results in \cite{CV2}, \cite{CV3}, \cite{GV} and \cite{Lit}), and we obtain the solution to the extremal problems for new interesting functions such as $|x|^{\alpha}$ for $-1 < \alpha$; \,$\log \,\bigl((x^2 + \alpha^2)/(x^2 + \beta^2)\bigr)$, for 
$0 \leq \alpha < \beta$;\, $\log\bigl(x^2 + \alpha^2\bigr)$; and $|x|^{2n} \log|x|$\,, for $n \in \N$. Some of the extremal $L^1(\R)$-approximations (without the one-sided condition) have previously been obtained by Sz.- Nagy (cf.\ \cite[Chapter 7]{Shapiro}). Further applications to number theory include optimal approximations of theta functions by trigonometric polynomials and optimal bounds for certain Hilbert-type inequalities related to the discrete 
Hardy-Littlewood-Sobolev inequality in dimension one.

\section*{Part I: The Gaussian}
\section{The Extremal Problem for the Gaussian}
For each positive value of $\lambda$ we define three entire functions as follows:
\begin{align}
K_{\lambda}(z) &= \Bigl(\frac{\cos \pi z}{\pi}\Bigr)
  \bigg\{\sum_{n=-\infty}^{\infty} (-1)^{n+1} \frac{G_{\lambda}\bigl(n+\hh\bigr)}{\bigl(z-n-\h\bigr)}\bigg\}\label{intro2},\\
L_{\lambda}(z) &= \Bigl(\frac{\cos \pi z}{\pi}\Bigr)^2 
	\bigg\{\sum_{m=-\infty}^{\infty} \frac{G_{\lambda}\bigl(m+\hh\bigr)}{\bigl(z-m-\h\bigr)^2} 
	+  \sum_{n=-\infty}^{\infty} \frac{G_{\lambda}^{\prime}\bigl(n+\h\bigr)}{\bigl(z-n-\h\bigr)}\bigg\}\label{intro3},\\
M_{\lambda}(z) &= \Bigl(\frac{\sin \pi z}{\pi}\Bigr)^2 
	\bigg\{\sum_{m=-\infty}^{\infty} \frac{G_{\lambda}(m)}{(z-m)^2} 
		+  \sum_{n=-\infty}^{\infty} \frac{G_{\lambda}^{\prime}(n)}{(z-n)}\bigg\}.\label{intro4}
\end{align}

The function $K_{\lambda}(z)$ is an entire function of exponential type $\pi$ which interpolates the values of 
the function $G_{\lambda}(z)$ at points of the coset $\Z + \hh$.  We will show that among all 
entire functions of exponential type at most $\pi$, the function $K_{\lambda}(z)$ provides the best approximation to 
$G_{\lambda}(z)$ with respect to the $L^1$-norm on $\R$.
  
The function $L_{\lambda}(z)$ is a real entire function of exponential type $2\pi$ which interpolates both the
values of $G_{\lambda}(z)$ and the values of its derivative $G_{\lambda}^{\prime}(z)$ on the coset
$\Z + \hh$.  Similarly, the function $M_{\lambda}(z)$ is a real entire function of exponential type $2\pi$ which
interpolates both the values of $G_{\lambda}(z)$ and the values of its derivative $G_{\lambda}^{\prime}(z)$ on the
integers $\Z$.   By a {\it real} entire function we understand an entire function whose restriction to $\R$ is real valued.
We will show that these functions satisfy the basic inequality
\begin{equation}\label{intro5}
L_{\lambda}(x) \le G_{\lambda}(x) \le M_{\lambda}(x)
\end{equation}
for all real $x$.  Moreover, we will show that the value of each of the two integrals 
\begin{equation*}\label{intro6}
\int_{-\infty}^{\infty} \Big\{G_{\lambda}(x) - L_{\lambda}(x)\Big\}\ \dx\quad\text{and}
 	\quad \int_{-\infty}^{\infty} \Big\{M_{\lambda}(x) - G_{\lambda}(x)\Big\}\ \dx,
\end{equation*}
is minimized.

In order to state a more precise form of our main results for the Gaussian function, we make use of the basic theta functions.
Here $v$ is a complex variable, $\tau$ is a complex variable with $\Im\{\tau\} > 0$, $q = e^{\pi i\tau}$,
and $e(z) = e^{2\pi iz}$.  Our notation for the theta functions follows that of Chandrasekharan \cite{KC}.
Thus we define
\begin{align}
\theta_1(v, \tau) &= \sum_{n= -\infty}^{\infty} q^{(n+\h)^2} e\bigl((n+\hh)v\bigr),\label{theta1}\\
\theta_2(v, \tau) &= \sum_{n= -\infty}^{\infty} (-1)^n q^{n^2} e(nv),\label{theta2}\\
\theta_3(v, \tau) &= \sum_{n= -\infty}^{\infty} q^{n^2} e(nv)\label{theta3}.
\end{align}
We note that for a fixed value of $\tau$ with 
$\Im\{\tau\} > 0$, each of the functions $v\mapsto \theta_1(v,\tau)$,
$v\mapsto \theta_2(v,\tau)$, and $v\mapsto \theta_3(v,\tau)$ is an {\it even} entire function of 
$v$.  The function $v\mapsto \theta_1(v, \tau)$ is periodic with period $2$, and satisfies the identity
\begin{equation}\label{theta4}
\theta_1(v + 1, \tau) = - \theta_1(v, \tau)
\end{equation}
for all complex $v$.  Both of the functions $v\mapsto \theta_2(v,\tau)$, and 
$v\mapsto \theta_3(v,\tau)$, are periodic with period $1$.  They are related by the identity
\begin{equation}\label{theta5}
\theta_2(v + \hh,\tau) = \theta_3(v,\tau).
\end{equation}
The transformation formulas for the theta functions (see \cite[Chapter V, Theorem 9, Corollary 1]{KC}) provide 
a connection with the Gaussian function $G_{\lambda}(z)$.  In particular we have
\begin{align}
\sum_{n=-\infty}^{\infty} (-1)^n G_{\lambda}(n - v) &= \lambda^{-\h} \theta_1\bigl(v, i\lambda^{-1}\bigr),\label{poisson1}\\
\sum_{n=-\infty}^{\infty} G_{\lambda}(n + \hh - v) &= \lambda^{-\h} \theta_2\bigl(v, i\lambda^{-1}\bigr),\label{poisson2}\\
\sum_{n=-\infty}^{\infty} G_{\lambda}(n - v) &= \lambda^{-\h} \theta_3\bigl(v, i\lambda^{-1}\bigr).\label{poisson3}
\end{align}

Our first main result identifies the entire function $K_{\lambda}(z)$ as the unique best approximation to $G_{\lambda}(z)$ 
on $\R$ among all entire functions of exponential type at most $\pi$.

\begin{theorem}\label{thm1}
Let $F(z)$ be an entire function of exponential type at most $\pi$.  Then
\begin{equation}\label{intro10}
\lambda^{-\h} \int_{-\h}^{\h} \theta_1\bigl(u,i\lambda^{-1}\bigr)\ \du 
	\le \int_{-\infty}^{\infty} \bigl|G_{\lambda}(x) - F(x)\bigr|\ \dx, 
\end{equation}
and there is equality in {\rm (\ref{intro10})} if and only if $F(z) = K_{\lambda}(z)$.
\end{theorem}

Next we consider the problem of minorizing $G_{\lambda}(z)$ on $\R$ by a real entire function of exponential type 
at most $2\pi$.  

\begin{theorem}\label{thm2}
Let $F(z)$ be a real entire function of exponential type at most $2\pi$ such that
\begin{equation*}\label{intro20}
F(x) \le G_{\lambda}(x)
\end{equation*}
for all real $x$.  Then
\begin{equation}\label{intro21}
\int_{-\infty}^{\infty} F(x)\ \dx \le \lambda^{-\h} \theta_2\bigl(0, i\lambda^{-1}\bigr),
\end{equation}
and there is equality in {\rm (\ref{intro21})} if and only if $F(z) = L_{\lambda}(z)$.
\end{theorem}

Here is the analogous result for the problem of majorizing $G_{\lambda}(z)$ on $\R$ by a real entire function of exponential type 
at most $2\pi$.  

\begin{theorem}\label{thm3} 
Let $F(z)$ be a real entire function of exponential type at most $2\pi$ such that
\begin{equation*}\label{intro22}
G_{\lambda}(x) \le F(x)
\end{equation*}
for all real $x$.  Then
\begin{equation}\label{intro23}
\lambda^{-\h} \theta_3\bigl(0, i\lambda^{-1}\bigr) \le \int_{-\infty}^{\infty} F(x)\ \dx,
\end{equation}
and there is equality in {\rm (\ref{intro23})} if and only if $F(z) = M_{\lambda}(z)$.
\end{theorem}

It follows from Theorem \ref{thm1} that for $\delta > 0$, the entire function $z \mapsto K_{\lambda\delta^{-2}}(\delta z)$ is the 
unique best $L^1$-approximation to $G_{\lambda}(x)$ by an entire function of exponential type $\pi\delta$.  In a similar manner, using 
Theorem \ref{thm2} and Theorem \ref{thm3}, one can check that the real entire
functions $z \mapsto L_{\lambda \delta^{-2}}(\delta z)$ and $z \mapsto M_{\lambda \delta^{-2}}(\delta z)$ are the 
unique extremal minorant and majorant, respectively, of exponential type $2\pi \delta$ for the function $G_{\lambda}(x)$.  

The entire function $K_{\lambda}(z)$ has exponential type $\pi$, and the restriction $x\mapsto K_{\lambda}(x)$ of this function 
to $\R$ is clearly integrable.  It follows that the Fourier transform
\begin{equation*}\label{intro30}
\tK_{\lambda}(t) = \int_{-\infty}^{\infty} K_{\lambda}(x) e(-xt)\ \dx
\end{equation*}
is a continuous function on $\R$, and is supported on the compact interval $[-\h, \h]$.  The entire functions $L_{\lambda}(z)$ 
and $M_{\lambda}(z)$ have exponential type $2\pi$, and the restrictions of these functions to $\R$ are both integrable.  Hence 
their Fourier transforms
\begin{equation*}\label{intro31}
\tL_{\lambda}(t) = \int_{-\infty}^{\infty} L_{\lambda}(x) e(-xt)\ \dx,
		\quad\text{and}\quad  \tM_{\lambda}(t) = \int_{-\infty}^{\infty} M_{\lambda}(x) e(-xt)\ \dx,
\end{equation*}
are both continuous, and both Fourier transforms are supported on the compact interval $[-1, 1]$.  These Fourier transforms can be 
given explicitly in terms of the theta functions.

\begin{theorem}\label{thm4}
If $-\h \le t \le \h$ then the Fourier transform $t\mapsto \tK_{\lambda}(t)$ is given by
\begin{equation}\label{intro32}
\tK_{\lambda}(t) = \theta_1(t, i\lambda).
\end{equation}
If $-1 \le t \le 1$ then the Fourier transforms $t\mapsto \tL_{\lambda}(t)$ and $t\mapsto \tM_{\lambda}(t)$ are given by
\begin{equation}\label{intro33}
\tL_{\lambda}(t) = (1 - |t|)\theta_1(t, i\lambda) - (2\pi)^{-1}\lambda \sgn(t)\frac{\partial\theta_1}{\partial t}(t, i\lambda),
\end{equation}
and
\begin{equation}\label{intro34}
\tM_{\lambda}(t) = (1 - |t|)\theta_3(t, i\lambda) - (2\pi)^{-1}\lambda \sgn(t)\frac{\partial\theta_3}{\partial t}(t, i\lambda).
\end{equation}
\end{theorem}

Theorem \ref{thm4} plainly provides an alternative representation for each of the entire functions $K_{\lambda}(z)$, $L_{\lambda}(z)$,
and $M_{\lambda}(z)$.  Because $t\mapsto \tK_{\lambda}(t)$ is continuous and has compact support, the Fourier inversion
formula and (\ref{intro32}) imply that
\begin{equation*}\label{intro35}
K_{\lambda}(z) = \int_{-\infty}^{\infty} \tK_{\lambda}(t) e(zt)\ \dt = \int_{-\h}^{\h} \theta_1(t, i\lambda) e(zt)\ \dt
\end{equation*}
for all complex $z$.  Analogous representations hold for the functions $L_{\lambda}(z)$ and $M_{\lambda}(z)$ using
(\ref{intro33}) and (\ref{intro34}).

The extremal entire functions that we have identified here can be used to determine corresponding extremal trigonometric 
polynomials associated to the theta functions.  Let $N$ be a non-negative integer.  We define
\begin{align}\label{intro36}
\begin{split}
k_{\lambda, N}(x) &= \lambda^{\h} \sum_{n=-\infty}^{\infty} K_{(2N+2)^{-2}\lambda}\bigl((2N+2)(x + n)\bigr) \\
	&= \lambda^{\h} (2N+2)^{-1} \sum_{n=-N}^N \tK_{(2N+2)^{-2}\lambda}\Bigl(\frac{n}{2N+2}\Bigr) e(nx),
\end{split}
\end{align}
where the equality of these sums follows from the Poisson summation formula and the fact that $\tK_{(2N+2)^{-2}\lambda}(t) = 0$
for $\h \le |t|$.  In particular, $k_{\lambda, N}(x)$ is a trigonometric polynomial of degree $N$ defined on the quotient group $\R/\Z$.
We will show that this trigonometric polynomial is the best approximation to the the theta function 
$x\mapsto \theta_3\bigr(x, i\lambda^{-1}\bigr)$ in $L^1$-norm on $\R/\Z$.

\begin{theorem}\label{thm5}
Let $p(x)$ be a trigonometric polynomial of degree at most $N$ defined on $\R/\Z$.  Then
\begin{equation}\label{intro40}
\int_{-\h}^{\h} \theta_1\bigl(u, i\lambda^{-1}(2N+2)^2\bigr)\ \du 
	\le \int_{\R/\Z} \bigl|\theta_3\bigl(x, i\lambda^{-1}\bigr) - p(x)\bigr|\ \dx,
\end{equation}
and there is equality in {\rm (\ref{intro40})} if and only if $p(x) = k_{\lambda, N}(x)$.
\end{theorem}

In a similar manner, we define
\begin{equation}\label{intro41}
\begin{split}
l_{\lambda, N}(x) &= \lambda^{\h} \sum_{n=-\infty}^{\infty} L_{(N+1)^{-2}\lambda}\bigl((N+1)(x + n)\bigr) \\
	&= \lambda^{\h} (N+1)^{-1} \sum_{n=-N}^N \tL_{(N+1)^{-2}\lambda}\Bigl(\frac{n}{N+1}\Bigr) e(nx),
\end{split}
\end{equation}
and
\begin{equation}\label{intro42}
\begin{split}
m_{\lambda, N}(x) &= \lambda^{\h} \sum_{n=-\infty}^{\infty} M_{(N+1)^{-2}\lambda}\bigl((N+1)(x + n)\bigr) \\
	&= \lambda^{\h} (N+1)^{-1} \sum_{n=-N}^N \tM_{(N+1)^{-2}\lambda}\Bigl(\frac{n}{N+1}\Bigr) e(nx).
\end{split}
\end{equation}
Again the identities in (\ref{intro41}) and (\ref{intro42}) follow from the Poisson summation formula and the fact that
$\tL_{(N+1)^{-2}\lambda}(t) = \tM_{(N+1)^{-2}\lambda}(t) = 0$ for $1 \le |t|$.  Both of the functions $l_{\lambda, N}(x)$ and
$m_{\lambda, N}(x)$ are real valued trigonometric polynomials of degree $N$ defined on the quotient group $\R/\Z$.  It follows
from (\ref{intro5}) and (\ref{poisson3}) that they satisfy the inequality
\begin{equation}\label{intro43}
l_{\lambda, N}(x) \le \theta_3\bigl(x, i\lambda^{-1}\bigr) \le m_{\lambda, N}(x)
\end{equation}
at each point $x$ in $\R/\Z$.  We will prove that these trigonometric polynomials are the extreme minorant and majorant
for the function $x\mapsto \theta_3\bigr(x, i\lambda^{-1}\bigr)$ on $\R/\Z$.

\begin{theorem}\label{thm6}
If $q(x)$ is a real valued trigonometric polynomial of degree at most $N$ such that
\begin{equation}\label{intro44}
q(x) \le \theta_3\bigl(x, i\lambda^{-1}\bigr)
\end{equation}
at each point $x$ in $\R/\Z$, then
\begin{equation}\label{intro45}
\int_{\R/\Z} q(x)\ \dx \le \theta_2\bigl(0, i\lambda^{-1}(N+1)^2\bigr).
\end{equation}
Moreover, there is equality in {\rm (\ref{intro45})} if and only if $q(x) = l_{\lambda, N}(x)$.
If $r(x)$ is a real valued trigonometric polynomial of degree at most $N$ such that
\begin{equation}\label{intro46}
\theta_3\bigl(x, i\lambda^{-1}\bigr) \le r(x)
\end{equation}
at each point $x$ in $\R/\Z$, then
\begin{equation}\label{intro47}
\theta_3\bigl(0, i\lambda^{-1}(N+1)^2\bigr) \le \int_{\R/\Z} r(x)\ \dx.
\end{equation}
Moreover, there is equality in {\rm (\ref{intro47})} if and only if $r(x) = m_{\lambda, N}(x)$.
\end{theorem}

\section{Integral Representations}

\begin{lemma}\label{lem1}
Let $z$ and $w$ be distinct complex numbers.  Then we have
\begin{align}\label{gauss1}
\begin{split}
\frac{G_{\lambda}(z) - G_{\lambda}(w)}{z - w}
	&= 2\pi \lambda^{\g} \int_{-\infty}^0 \int_{-\infty}^0 e^{-2\pi\lambda tu} G_{\lambda}(z - t)G_{\lambda}(w-u)\ \du\ \dt \\
        &\qquad - 2\pi \lambda^{\g} \int_0^{\infty} \int_0^{\infty} e^{-2\pi\lambda tu} G_{\lambda}(z - t)G_{\lambda}(w-u)\ \du\ \dt.
\end{split}
\end{align}
\end{lemma}

\begin{proof}
It suffices to prove the identity (\ref{gauss1}) for $\lambda = 1$, then the general case will follow from an elementary change of 
variables.  Therefore we simplify our notation and write $G(z) = G_1(z)$.  We note that $G(z)$ satisfies the identity
\begin{equation}\label{gauss2}
G(z)^{-1} = \int_{-\infty}^{\infty} e^{2\pi zt} G(t)\ \dt
\end{equation}
for all complex numbers $z$, and the identity
\begin{equation}\label{gauss3}
G(z)G(w) e^{2\pi zw} = G(z - w)
\end{equation}
for all pairs of complex numbers $z$ and $w$.  From (\ref{gauss2}) we get
\begin{align}\label{gauss4}
\begin{split}
\frac{G(z) - G(w)}{z - w} &= G(z)G(w)\bigg\{\frac{G(w)^{-1} - G(z)^{-1}}{z - w}\bigg\} \\ 
	&= G(z)G(w)(z - w)^{-1} \int_{-\infty}^{\infty} \big\{e^{2\pi wt} - e^{2\pi zt}\big\} G(t)\ \dt.
\end{split}
\end{align}
Then using Fubini's theorem we find that
\begin{align}\label{gauss5}
\begin{split}
(z - w)^{-1}\int_{-\infty}^{\infty} &\big\{e^{2\pi wt} - e^{2\pi zt}\big\} G(t)\ \dt \\
   &= 2\pi \int_{-\infty}^0 \bigg\{\int_t^0 e^{2\pi(z - w)u}\ \du\bigg\} e^{2\pi wt} G(t)\ \dt \\
   &\qquad - 2\pi \int_0^{\infty} \bigg\{\int_0^t e^{2\pi(z - w)u}\ \du\bigg\} e^{2\pi wt} G(t)\ \dt \\
   &= 2\pi \int_{-\infty}^0 \bigg\{\int_{-\infty}^u e^{2\pi wt} G(t)\ \dt\bigg\} e^{2\pi(z - w)u}\ \du \\
   &\qquad - 2\pi \int_0^{\infty} \bigg\{\int_u^{\infty} e^{2\pi wt} G(t)\ \dt\bigg\} e^{2\pi(z - w)u}\ \du \\
   &= 2\pi \int_{-\infty}^0 \bigg\{\int_{-\infty}^0 e^{2\pi w(t+u)} G(t+u)\ \dt\bigg\} e^{2\pi(z - w)u}\ \du \\
   &\qquad - 2\pi \int_0^{\infty} \bigg\{\int_0^{\infty} e^{2\pi w(t+u)} G(t+u)\ \dt\bigg\} e^{2\pi(z - w)u}\ \du \\
   &= 2\pi \int_{-\infty}^0 \int_{-\infty}^0 e^{2\pi (wt + zu)} G(t+u)\ \dt\ \du \\
   &\qquad - 2\pi \int_0^{\infty} \int_0^{\infty} e^{2\pi (wt + zu)} G(t+u)\ \dt\ \du. \\
\end{split}
\end{align}
Next we apply (\ref{gauss3}) twice and get
\begin{align}\label{gauss6}
\begin{split}
G(z)G(w)e^{2\pi (wt + zu)} G(t+u) &= G(z)G(w)G(u)G(t)e^{-2\pi tu + 2\pi wt + 2\pi zu} \\
                                  &= G(z - u)G(w - t) e^{-2\pi tu}.
\end{split}
\end{align}
Then we combine (\ref{gauss4}), (\ref{gauss5}) and (\ref{gauss6}) to obtain the special case
\begin{align}\label{gauss7}
\begin{split}
\frac{G(z) - G(w)}{z - w}
	&= 2\pi \int_{-\infty}^0 \int_{-\infty}^0 e^{-2\pi tu} G(z - t)G(w-u)\ \du\ \dt \\
        &\qquad - 2\pi \int_0^{\infty} \int_0^{\infty} e^{-2\pi tu} G(z - t)G(w-u)\ \du\ \dt.
\end{split}
\end{align}
The more general identity (\ref{gauss1}) follows by replacing $z$ with $\lambda^{\h}z$, by 
replacing $w$ with $\lambda^{\h}w$, and by making a corresponding change of variables in each integral 
on the right of (\ref{gauss7}).
\end{proof}

\begin{lemma}\label{lem2}
Let $z$ and $w$ be distinct complex numbers.  Then we have
\begin{align}\label{gauss10}
\begin{split}
&\frac{G_{\lambda}(z)}{(z - w)^2} - \frac{G_{\lambda}(w)}{(z - w)^2} - \frac{G_{\lambda}^{\prime}(w)}{z - w}\\
	&= (2\pi)^2 \lambda^{\f} \int_{-\infty}^0 \int_{-\infty}^0 t e^{-2\pi\lambda tu} G_{\lambda}(z - t)
		\big\{G_{\lambda}(w) - G_{\lambda}(w-u)\big\}\ \du\ \dt \\
        &\quad -(2\pi)^2 \lambda^{\f} \int_0^{\infty} \int_0^{\infty} t e^{-2\pi\lambda tu} G_{\lambda}(z - t)
		\big\{G_{\lambda}(w) - G_{\lambda}(w-u)\big\}\ \du\ \dt.
\end{split}
\end{align}
\end{lemma}

\begin{proof}
We differentiate both sides of (\ref{gauss1}) with respect to $w$ and obtain the identity
\begin{align}\label{gauss11}
\begin{split}
&\frac{G_{\lambda}(z)}{(z - w)^2} - \frac{G_{\lambda}(w)}{(z - w)^2} - \frac{G_{\lambda}^{\prime}(w)}{z - w}\\
	&= 2\pi \lambda^{\g} \int_{-\infty}^0 \int_{-\infty}^0 e^{-2\pi\lambda tu} G_{\lambda}(z - t)
		G_{\lambda}^{\prime}(w-u)\ \du\ \dt \\
        &\quad - 2\pi \lambda^{\g} \int_0^{\infty} \int_0^{\infty} e^{-2\pi\lambda tu} G_{\lambda}(z - t)
		G_{\lambda}^{\prime}(w-u)\ \du\ \dt.
\end{split}\end{align}
Using integration by parts we get
\begin{align}\label{gauss12}
\begin{split}
\int_{-\infty}^0 e^{-2\pi\lambda tu}&G_{\lambda}^{\prime}(w-u)\ \du\\
	&= 2\pi\lambda \int_{-\infty}^0 t e^{-2\pi\lambda tu}
		\big\{G_{\lambda}(w) - G_{\lambda}(w-u)\big\}\ \du,
\end{split}
\end{align}
and
\begin{align}\label{gauss13}
\begin{split}
\int_0^{\infty} e^{-2\pi\lambda tu}&G_{\lambda}^{\prime}(w-u)\ \du\\
	&= 2\pi\lambda \int_0^{\infty} t e^{-2\pi\lambda tu}
		\big\{G_{\lambda}(w) - G_{\lambda}(w-u)\big\}\ \du.
\end{split}
\end{align}
The corollary follows now by combining (\ref{gauss11}), (\ref{gauss12}) and (\ref{gauss13}). 
\end{proof}

In order to apply the identities (\ref{poisson1}), (\ref{poisson2}) and (\ref{poisson3}), we require simple estimates for
certain partial sums.

\begin{lemma}\label{lem3}
For all real $u$ and positive integers $N$, we have
\begin{align}
\sum_{n=-N-1}^N (-1)^n G_{\lambda}(n + \hh - u) &\ll_{\lambda} \min\{1, |u|\},\label{poisson4}\\
\sum_{n=-N-1}^N \big\{G_{\lambda}(n + \hh) - G_{\lambda}(n + \hh - u)\big\} &\ll_{\lambda} \min\{1, |u|\},\label{poisson5}\\
\sum_{n=-N}^N \big\{G_{\lambda}(n) - G_{\lambda}(n - u)\big\} &\ll_{\lambda} \min\{1, |u|\},\label{poisson6}
\end{align}
where the constant implied by $\ll_{\lambda}$ depends on $\lambda$, but not on $u$ or $N$.
\end{lemma}

\begin{proof}
For each positive integer $N$,
\begin{equation*}\label{gauss14}
u\mapsto S_{\lambda, N}(u) = \sum_{n=-N-1}^N (-1)^n G_{\lambda}(n + \hh - u)
\end{equation*}
is an odd function of $u$.  Hence its derivative is an even function of $u$.  Therefore we get
\begin{align*}\label{gauss15}
\bigl|S_{\lambda, N}(u)\bigr| &= \biggl|\int_0^u S_{\lambda, N}^{\prime}(v)\ \dv\biggr|\\
	&\le \int_0^{|u|} \bigg\{\sum_{n=-\infty}^{\infty} \bigl|G_{\lambda}^{\prime}(n + \hh - v)\bigr|\bigg\}\ \dv\\
	&\le C_{\lambda} |u|,
\end{align*}
where
\begin{equation*}\label{gauss16}
C_{\lambda} = \sup_{v\in\R} \bigg\{\sum_{n=-\infty}^{\infty} \bigl|G_{\lambda}^{\prime}(n + \hh - v)\bigr|\bigg\}
\end{equation*}
is obviously finite.  We also have
\begin{equation*}\label{gauss17}
\bigl|S_{\lambda, N}(u)\bigr| \le \sup_{v\in\R} \bigg\{\sum_{n=-\infty}^{\infty} \bigl|G_{\lambda}(n + \hh - v)\bigr|\bigg\} < \infty,
\end{equation*}
and the bound (\ref{poisson4}) follows.

The proofs of (\ref{poisson5}) and (\ref{poisson6}) are very similar.
\end{proof}

We have noted that the entire function $z\mapsto G_{\lambda}(z) - K_{\lambda}(z)$ vanishes at each point of 
the coset $\Z + \hh$.  It follows that
\begin{equation*}\label{gauss20}
z\mapsto \frac{\pi}{\cos \pi z}\Big\{G_{\lambda}(z) - K_{\lambda}(z)\Big\}
\end{equation*}
is an entire function.

\begin{lemma}\label{lem4}
For all complex $z$ we have
\begin{align}\label{gauss21}
\begin{split}
\frac{\pi}{\cos \pi z}&\Big\{G_{\lambda}(z) - K_{\lambda}(z)\Big\}\\
	&= \pi\lambda \int_{-\infty}^{\infty}\frac{G_{\lambda}(z-t)}{\cosh \pi\lambda t}
		\int_{-\h}^{\h} \cosh 2\pi\lambda tu\ \theta_1\bigl(u, i\lambda^{-1}\bigr)\ \du\ \dt.
\end{split}
\end{align}
\end{lemma}

\begin{proof}
We use the partial fraction expansion
\begin{equation}\label{gauss22}
\lim_{N\rightarrow \infty} \sum_{n=-N-1}^N \frac{(-1)^{n+1}}{z - n - \hh} = \frac{\pi}{\cos \pi z},
\end{equation}
which converges uniformly on compact subsets of $\C\setminus\{\Z + \h\}$.  Then it follows from (\ref{intro2}) and (\ref{gauss22}) that
\begin{align}\label{gauss23}
\begin{split}
\frac{\pi}{\cos \pi z}&\Big\{G_{\lambda}(z) - K_{\lambda}(z)\Big\}\\
&= \lim_{N\rightarrow \infty} \sum_{n=-N-1}^N (-1)^{n+1}\bigg\{\frac{G_{\lambda}(z) - G_{\lambda}(n + \hh)}{z - n - \hh}\bigg\}.
\end{split}
\end{align}
As the function on the left of (\ref{gauss23}) is entire and a compact subset of $\C$ intersects $\Z + \h$ in finitely many points, we
find that the limit on the right of (\ref{gauss23}) converges uniformly on compact subsets of $\C$.

For positive integers $N$ and all real $u$ let
\begin{equation*}\label{gauss24}
S_{\lambda, N}(u) = \sum_{n=-N-1}^N (-1)^n G_{\lambda}(n + \hh - u).
\end{equation*}
Then (\ref{poisson1}) implies that
\begin{equation}\label{gauss25}
\lim_{N\rightarrow \infty} S_{\lambda, N}(u) = \lambda^{-\h} \theta_1\bigl(u - \hh, i\lambda^{-1}\bigr).
\end{equation}
We use the identity (\ref{gauss1}) with $w = n + \h$ and sum over integers $n$ satisfying $-N-1 \le n \le N$.  We find that
\begin{align}\label{gauss26}
\begin{split}
 \sum_{n=-N-1}^N (-1)^{n+1}&\bigg\{\frac{G_{\lambda}(z) - G_{\lambda}(n + \hh)}{z - n - \hh}\bigg\}\\
 	&= 2\pi \lambda^{\g} \int_0^{\infty} \int_0^{\infty} e^{-2\pi\lambda tu} G_{\lambda}(z - t)S_{\lambda, N}(u)\ \du\ \dt \\
        &\qquad - 2\pi \lambda^{\g} \int_{-\infty}^0 \int_{-\infty}^0 e^{-2\pi\lambda tu} G_{\lambda}(z - t)S_{\lambda, N}(u)\ \du\ \dt. 
\end{split}
\end{align}
Next we let $N\rightarrow \infty$ on both sides of (\ref{gauss26}).  The limit on the left hand side is determined by (\ref{gauss23}).
On the right hand side we use (\ref{poisson4}) and the dominated convergence theorem to move the limit inside the integral.
Then we use (\ref{gauss25}).  In this way we arrive at the identity
\begin{align}\label{gauss27}
\begin{split}
\frac{\pi}{\cos \pi z}&\Big\{G_{\lambda}(z) - K_{\lambda}(z)\Big\}\\
	&= 2\pi \lambda \int_0^{\infty} \int_0^{\infty} e^{-2\pi\lambda tu} G_{\lambda}(z - t)\theta_1\bigl(u - \hh, i\lambda^{-1}\bigr)\ \du\ \dt \\
        &\qquad - 2\pi \lambda \int_{-\infty}^0 \int_{-\infty}^0 e^{-2\pi\lambda tu} G_{\lambda}(z - t)\theta_1\bigl(u - \hh, i\lambda^{-1}\bigr)\ \du\ \dt.
\end{split}
\end{align}

If $0 < t$ then, using (\ref{theta4}) and the fact that $u\mapsto \theta_1\bigl(u, i\lambda^{-1}\bigr)$ is an even function, we get
\begin{align}\label{gauss28}
\begin{split}
\int_0^{\infty} e^{-2\pi\lambda tu}&\theta_1\bigl(u - \hh, i\lambda^{-1}\bigr)\ \du\\
	&= \sum_{m=0}^{\infty} \int_0^1 e^{-2\pi\lambda t(u + m)} \theta_1\bigl(u +m - \hh, i\lambda^{-1}\bigr)\ \du\\
	&= \sum_{m=0}^{\infty} (-1)^m e^{-2\pi\lambda tm} \int_0^1e^{-2\pi\lambda tu} \theta_1\bigl(u - \hh, i\lambda^{-1}\bigr)\ \du\\
	&= \big\{e^{\pi\lambda t} + e^{-\pi\lambda t}\big\}^{-1} \int_{-\h}^{\h}e^{-2\pi\lambda tu} \theta_1\bigl(u, i\lambda^{-1}\bigr)\ \du\\	
	&= \big\{2\cosh \pi\lambda t\big\}^{-1} \int_{-\h}^{\h} \cosh 2\pi\lambda tu \ \theta_1\bigl(u, i\lambda^{-1}\bigr)\ \du.	
\end{split}
\end{align}
If $t < 0$ then in a similar manner we find that
\begin{align}\label{gauss29}
\begin{split}
\int_{-\infty}^0 e^{-2\pi\lambda tu}&\theta_1\bigl(u - \hh, i\lambda^{-1}\bigr)\ \du\\
	&= -\big\{2\cosh \pi\lambda t\big\}^{-1} \int_{-\h}^{\h} \cosh 2\pi\lambda tu \ \theta_1\bigl(u, i\lambda^{-1}\bigr)\ \du.	
\end{split}
\end{align}
The identity (\ref{gauss21}) follows now by combining (\ref{gauss27}), (\ref{gauss28}) and (\ref{gauss29}).
\end{proof}

Because $z\mapsto L_{\lambda}(z)$ interpolates both the value of $G_{\lambda}(z)$ and the value of its 
derivative $G_{\lambda}^{\prime}(z)$ at each point of the coset $\Z + \hh$, the entire function
\begin{equation*}\label{gauss34}
z\mapsto G_{\lambda}(z) - L_{\lambda}(z)
\end{equation*}
has a zero of multiplicity at least $2$ at each point of 
$\Z + \h$.  It follows that
\begin{equation*}\label{gauss35}
z\mapsto \biggl(\frac{\pi}{\cos \pi z}\biggr)^2\Big\{G_{\lambda}(z) - L_{\lambda}(z)\Big\}
\end{equation*}
is an entire function.  In a similar manner, we find that
\begin{equation*}\label{gauss36}
z\mapsto \biggl(\frac{\pi}{\sin \pi z}\biggr)^2\Big\{M_{\lambda}(z) - G_{\lambda}(z)\Big\}
\end{equation*}
is an entire function.

\begin{lemma}\label{lem5}
For all complex $z$ we have
\begin{align}\label{gauss37}
\begin{split}
&\biggl(\frac{\pi}{\cos \pi z}\biggr)^2\Big\{G_{\lambda}(z) - L_{\lambda}(z)\Big\}\\
       &= 2\pi^2\lambda^2 \int_{-\infty}^{\infty} \frac{t G_{\lambda}(z - t)}{\sinh \pi\lambda t}
	\int_{-\h}^{\h} e^{-2\pi\lambda tu}\big\{\theta_3\bigl(u,i\lambda^{-1}\bigr) - \theta_3\bigl(\hh,i\lambda^{-1}\bigr)\big\}\ \du\ \dt,
\end{split}
\end{align}
and
\begin{align}\label{gauss38}
\begin{split}
&\biggl(\frac{\pi}{\sin \pi z}\biggr)^2\Big\{M_{\lambda}(z) - G_{\lambda}(z)\Big\}\\
	&= 2\pi^2\lambda^2 \int_{-\infty}^{\infty} \frac{t G_{\lambda}(z - t)}{\sinh \pi\lambda t}
	\int_{-\h}^{\h} e^{-2\pi\lambda tu}\big\{\theta_2\bigl(\hh,i\lambda^{-1}\bigr) - \theta_2\bigl(u,i\lambda^{-1}\bigr)\big\}\ \du\ \dt.
\end{split}
\end{align}
\end{lemma}

\begin{proof}
In order to establish (\ref{gauss37}) we use the partial fraction expansion
\begin{equation}\label{gauss39}
\lim_{N\rightarrow \infty} \sum_{n=-N-1}^N \frac{1}{\bigl(z - n - \h\bigr)^2} = \biggl(\frac{\pi}{\cos \pi z}\biggr)^2,
\end{equation}
which converges uniformly on compact subsets of $\C\setminus\big\{\Z + \h\big\}$.  Then it follows from (\ref{intro3}) and (\ref{gauss39}) that
\begin{align}\label{gauss40}
\begin{split}
\biggl(\frac{\pi}{\cos \pi z}&\biggr)^2\Big\{G_{\lambda}(z) - L_{\lambda}(z)\Big\}\\
	&= \lim_{N\rightarrow \infty} \sum_{n=-N-1}^N \bigg\{\frac{G_{\lambda}(z)}{(z - n - \h)^2} - \frac{G_{\lambda}(n + \h)}{(z - n - \h)^2} 
		- \frac{G_{\lambda}^{\prime}(n + \h)}{z - n - \h}\bigg\}.
\end{split}
\end{align}
As in the proof of Lemma \ref{lem4}, the limit on the right of (\ref{gauss40}) converges uniformly on compact subsets of $\C$.

For positive integers $N$ and all real $u$ let
\begin{equation*}\label{gauss41}
T_{\lambda, N}(u) = \sum_{n=-N-1}^N \big\{G_{\lambda}(n + \hh) - G_{\lambda}(n + \hh -u)\big\}.
\end{equation*}
From (\ref{poisson2}) we conclude that
\begin{equation}\label{gauss42}
\lim_{N\rightarrow \infty} T_{\lambda, N}(u) 
	= \lambda^{-\h}\big\{\theta_2\bigl(0, i\lambda^{-1}\bigr) - \theta_2\bigl(u, i\lambda^{-1}\bigr)\big\}.
\end{equation}
We use the identity (\ref{gauss10}) with $w = n+ \h$ and sum over integers $n$ satisfying $-N-1 \le n \le N$.  We get
\begin{align}\label{gauss43}
\begin{split}
\sum_{n=-N-1}^N&\bigg\{\frac{G_{\lambda}(z)}{(z - n - \h)^2} - \frac{G_{\lambda}(n + \h)}{(z - n - \h)^2} 
		- \frac{G_{\lambda}^{\prime}(n + \h)}{z - n - \h}\bigg\}\\
	&= (2\pi)^2 \lambda^{5/2} \int_{-\infty}^0 \int_{-\infty}^0 t e^{-2\pi\lambda tu} G_{\lambda}(z - t)T_{\lambda, N}(u)\ \du\ \dt \\
        &\qquad - (2\pi)^2 \lambda^{5/2} \int_0^{\infty} \int_0^{\infty} t e^{-2\pi\lambda tu} G_{\lambda}(z - t)T_{\lambda, N}(u)\ \du\ \dt. 
\end{split}
\end{align}
As in the proof of Lemma \ref{lem4}, we let $N\rightarrow \infty$ on both sides of (\ref{gauss43}).  The limit on the left hand side
is determined by (\ref{gauss40}).  On the right hand side we use (\ref{poisson5}), the dominated convergence theorem and
(\ref{gauss42}).  In this way we obtain the identity
\begin{align}\label{gauss44}
\begin{split}
&\biggl(\frac{\pi}{\cos \pi z}\biggr)^2\Big\{G_{\lambda}(z) - L_{\lambda}(z)\Big\}\\
	&= (2\pi \lambda)^2 \int_{-\infty}^0 \int_{-\infty}^0 t e^{-2\pi\lambda tu} G_{\lambda}(z - t)\big\{\theta_2\bigl(0, i\lambda^{-1}\bigr) 
		- \theta_2\bigl(u, i\lambda^{-1}\bigr)\big\}\ \du\ \dt \\
       &\quad - (2\pi \lambda)^2 \int_0^{\infty} \int_0^{\infty} t e^{-2\pi\lambda tu} G_{\lambda}(z - t)\big\{\theta_2\bigl(0, i\lambda^{-1}\bigr) 
        		- \theta_2\bigl(u, i\lambda^{-1}\bigr)\big\}\ \du\ \dt. 
\end{split}
\end{align}

If $0 < t$ then, using that $v \mapsto \theta_2(v,\tau)$ has period $1$ and (\ref{theta5}), we get
\begin{align}\label{gauss45}
\begin{split}
\int_0^{\infty}&e^{-2\pi\lambda tu}\big\{\theta_2\bigl(0, i\lambda^{-1}\bigr) - \theta_2\bigl(u, i\lambda^{-1}\bigr)\big\}\ \du\\
	&= \sum_{m=0}^{\infty} \int_0^1 e^{-2\pi\lambda t(u + m)} \big\{\theta_2\bigl(0, i\lambda^{-1}\bigr) 
		- \theta_2\bigl(u+m, i\lambda^{-1}\bigr)\big\}\ \du\\
	&= \big\{1 - e^{-2\pi\lambda t}\big\}^{-1} \int_0^1 e^{-2\pi\lambda tu} \big\{\theta_2\bigl(0, i\lambda^{-1}\bigr) 
		- \theta_2\bigl(u, i\lambda^{-1}\bigr)\big\}\ \du\\
	&= \big\{2\sinh \pi\lambda t\big\}^{-1} \int_{-\h}^{\h} e^{-2\pi\lambda tu} \big\{\theta_3\bigl(\hh, i\lambda^{-1}\bigr) 
		- \theta_3\bigl(u, i\lambda^{-1}\bigr)\big\}\ \du.
\end{split}
\end{align}
If $t < 0$ then in a similar manner we find that
\begin{align}\label{gauss46}
\begin{split}
\int_{-\infty}^0&e^{-2\pi\lambda tu}\big\{\theta_2\bigl(0, i\lambda^{-1}\bigr) - \theta_2\bigl(u, i\lambda^{-1}\bigr)\big\}\ \du\\
	&= -\big\{2\sinh \pi\lambda t\big\}^{-1} \int_{-\h}^{\h} e^{-2\pi\lambda tu} \big\{\theta_3\bigl(\hh, i\lambda^{-1}\bigr) 
		- \theta_3\bigl(u, i\lambda^{-1}\bigr)\big\}\ \du.
\end{split}
\end{align}
The identity (\ref{gauss37}) follows now by combining (\ref{gauss44}), (\ref{gauss45}) and (\ref{gauss46}).

The proof of (\ref{gauss38}) proceeds along the same lines using (\ref{poisson3}) and (\ref{poisson6}). We leave the details 
to the reader.
\end{proof}

\begin{corollary}\label{cor9}
For all real values of $x$ we have
\begin{equation}\label{gauss47}
0 < \biggl(\frac{\pi}{\cos \pi x}\biggr)^2\Big\{G_{\lambda}(x) - L_{\lambda}(x)\Big\},
\end{equation}
and
\begin{equation}\label{gauss48}
0 < \biggl(\frac{\pi}{\sin \pi x}\biggr)^2\Big\{M_{\lambda}(x) - G_{\lambda}(x)\Big\}.
\end{equation}
In particular, the inequality {\rm (\ref{intro5})} holds for all real $x$.
\end{corollary}

\begin{proof}
For real $u$ the periodic function $u\mapsto \theta_3\bigl(u, i\lambda^{-1}\bigr)$ takes its maximum value at
$u = 0$ and its minimum values at $u = \h$.  Therefore the function
\begin{equation*}\label{gauss49}
t\mapsto \int_{-\h}^{\h} e^{-2\pi\lambda tu}\big\{\theta_3\bigl(u,i\lambda^{-1}\bigr) - \theta_3\bigl(\hh,i\lambda^{-1}\bigr)\big\}\ \du,
\end{equation*}
which appears in the integrand on the right of (\ref{gauss37}), is positive for all real values of $t$.  This plainly verifies the
inequality (\ref{gauss47}). 

In a similar manner using (\ref{theta5}), the periodic function $u\mapsto \theta_2\bigl(u, i\lambda^{-1}\bigr)$ takes its
maximum value at $u = \h$ and its minimum value at $u = 0$.  Hence the function
\begin{equation*}\label{gauss50}
t\mapsto \int_{-\h}^{\h} e^{-2\pi\lambda tu}\big\{\theta_2\bigl(\hh,i\lambda^{-1}\bigr) - \theta_2\bigl(u,i\lambda^{-1}\bigr)\big\}\ \du,
\end{equation*}
which appears in the integrand on the right of (\ref{gauss38}), is positive for all real values of $t$.  This establishes the
inequality (\ref{gauss48}).
\end{proof}

\section{Proof of Theorem \ref{thm1}}

We recall that $v\mapsto \theta_1(v, i\lambda)$ is an entire function of the complex variable $v$.
The product formula for this theta function (see \cite[Chapter V, Theorem 6]{KC}) provides the representation
\begin{align}\label{theta6}
\begin{split}
\theta_1\bigl(v,i\lambda^{-1}\bigr)
	&= 2 e^{-\pi(4\lambda)^{-1}} \prod_{l=1}^{\infty}\bigl(1 - e^{-2\pi\lambda^{-1}l}\bigr) \cos \pi v \\
        &\qquad\qquad \prod_{m=1}^{\infty} \bigl(1 + e^{-2\pi\lambda^{-1}m + 2\pi iv}\bigr) 
            \prod_{n=1}^{\infty} \bigl(1 + e^{-2\pi\lambda^{-1}n - 2\pi iv}\bigr).
\end{split}
\end{align}
It follows from (\ref{theta6}) that the only real zeros of $\theta_1\bigl(v, i\lambda^{-1}\bigr)$ are zeros 
of $\cos \pi v$.  That is, the only real zeros are simple zeros at the points of $\Z + \h$.
Because $\theta_1\bigl(0, i\lambda^{-1}\bigr) > 0$, it follows that $\theta_1\bigl(u, i\lambda^{-1}\bigr) > 0$ for all 
real values of $u$ in the open interval $-\h < u < \h$.  This implies that the integral on the left of (\ref{intro10}) is 
positive.   Also, the integral on the right of (\ref{gauss21}) is positive for real values of $z = x$.  Alternatively, we have  
\begin{equation*}\label{theta7}
\frac{\pi}{\cos \pi x}\Big\{G_{\lambda}(x) - K_{\lambda}(x)\Big\} > 0
\end{equation*}
for all real $x$, and therefore
\begin{equation}\label{theta8}
\sgn\Big\{G_{\lambda}(x) - K_{\lambda}(x)\Big\} = \sgn(\cos \pi x)
\end{equation}
for all real $x$.

From the series expansion 
(\ref{theta1}) we find that
\begin{equation}\label{pf00}
\lambda^{-\h} \int_{-\h}^{\h} \theta_1(u, i\lambda^{-1})\ \du
	= \frac{1}{\pi} \sum_{n=-\infty}^{\infty} \frac{(-1)^n}{n+\hh} \tG_{\lambda}(n + \hh),
\end{equation}
where
\begin{equation*}\label{pf0}
\tG_{\lambda}(t) = \lambda^{-\h} e^{-\pi\lambda^{-1} t^2} = \int_{-\infty}^{\infty} G_{\lambda}(x) e(-xt)\ \dx
\end{equation*}
is the Fourier transform of $G_{\lambda}(x)$ on $\R$.
Now let $F(z)$ be an entire function of exponential type at most $\pi$.  Without loss of generality we may assume that
\begin{equation*}\label{pf1}
\int_{-\infty}^{\infty} \bigl|G_{\lambda}(x) - F(x)\bigr|\ \dx < \infty.
\end{equation*}
It follows that $F$ is integrable on $\R$ and therefore the Fourier transform
\begin{equation*}\label{pf2}
t\mapsto \tF(t) = \int_{-\infty}^{\infty} F(x) e(-tx)\ \dx
\end{equation*}
is continuous on $\R$, and supported on $\bigl[-\h, \h\bigr]$.  The function
\begin{equation*}\label{pf3}
x\mapsto \sgn(\cos \pi x)
\end{equation*}
is periodic on $\R$ with period $2$, and has the Fourier expansion
\begin{equation}\label{pf4}
\sgn(\cos \pi x) = \lim_{N\rightarrow \infty} \frac{1}{\pi} \sum_{n=-N-1}^{N} \frac{(-1)^n}{n+\hh} e\bigl((n+\hh)x\bigr).
\end{equation}
Moreover, the partial sums on the right of (\ref{pf4}) are uniformly bounded, and therefore
\begin{align}\label{pf5}
\begin{split}
\int_{-\infty}^{\infty} &\sgn(\cos \pi x) \big\{G_{\lambda}(x) - F(x)\big\}\ \dx\\
	&= \lim_{N\rightarrow\infty} \frac{1}{\pi} \sum_{n = -N-1}^{N} 
		\frac{(-1)^n}{n+\hh} \int_{-\infty}^{\infty} \big\{G_{\lambda}(x) - F(x)\big\} e\bigl((n+\hh)x\bigr)\ \dx\\
	&= \lim_{N\rightarrow\infty} \frac{1}{\pi} \sum_{n = -N-1}^{N}  \frac{(-1)^n}{n+\hh} \big\{\tG_{\lambda}\bigl(-n-\hh\bigr) - \tF\bigl(-n-\hh\bigr)\big\}\\	
	&= \frac{1}{\pi} \sum_{n = -\infty}^{\infty} \frac{(-1)^n}{n+\hh} \tG_{\lambda}\bigl(n+\hh\bigr).
\end{split}
\end{align}
It is clear from (\ref{pf00}) and (\ref{pf5}) that
\begin{equation}\label{pf6}
\lambda^{-\h} \int_{-\h}^{\h} \theta_1\bigl(u, i\lambda^{-1}\bigr)\ \du \le \int_{-\infty}^{\infty} \bigl|G_{\lambda}(x) - F(x)\bigr|\ \dx,
\end{equation}
and this verifies (\ref{intro10}).  Then (\ref{theta8}), (\ref{pf00}) and (\ref{pf5}) lead to the identity
\begin{align}\label{pf7}
\begin{split}
\lambda^{-\h} \int_{-\h}^{\h} \theta_1\bigl(u, i\lambda^{-1}\bigr)\ \du 
	&= \int_{-\infty}^{\infty} \sgn(\cos \pi x) \big\{G_{\lambda}(x) - K_{\lambda}(x)\big\}\ \dx\\
	&= \int_{-\infty}^{\infty} \bigl|G_{\lambda}(x) - K_{\lambda}(x)\bigr|\ \dx.
\end{split}
\end{align}
Plainly (\ref{pf7}) shows that there is equality in the inequality (\ref{intro10}) in case $F(z) = K_{\lambda}(z)$.

Finally, we {\it assume} that $F(z)$ is an entire function of exponential type at most $\pi$ for which there is equality in the 
inequality (\ref{intro10}).  Then (\ref{pf00}) and (\ref{pf5}) imply that
\begin{equation}\label{pf8}
\int_{-\infty}^{\infty} \sgn(\cos \pi x) \big\{G_{\lambda}(x) - F(x)\big\}\ \dx = \int_{-\infty}^{\infty} \bigl|G_{\lambda}(x) - F(x)\bigr|\ \dx.
\end{equation}
As $x\mapsto G_{\lambda}(x) - F(x)$ is continuous, we conclude from (\ref{pf8}) that
\begin{equation*}\label{pf9}
\sgn(\cos \pi x) \big\{G_{\lambda}(x) - F(x)\big\} = \bigl|G_{\lambda}(x) - F(x)\bigr|
\end{equation*}
for all real $x$.  This implies that
\begin{equation*}\label{pf10}
G_{\lambda}\bigl(n + \hh\bigr) = K_{\lambda}\bigl(n + \hh\bigr) = F\bigl(n + \hh\bigr)
\end{equation*}
for each integer $n$.  Therefore
\begin{equation}\label{pf11}
z\mapsto K_{\lambda}(z) - F(z)
\end{equation}
is an entire function of exponential type at most $\pi$ and takes the value zero at each point of the set $\Z + \hh$.  From
basic interpolation theorems for entire functions of exponential type (see \cite[Vol. II, p. 275]{Z}), we conclude that the
entire function (\ref{pf11}) is identically zero. This completes the proof of Theorem 1.

\section{Proofs of Theorems \ref{thm2} and \ref{thm3}}

Let $F(z)$ be an entire function of exponential type at most $2\pi$ such that
\begin{equation}\label{pf19}
F(x) \le G_{\lambda}(x)
\end{equation}
for all real $x$.  Clearly we may assume that $x\mapsto F(x)$ is integrable on $\R$, for if not then (\ref{intro21}) is trivial.
Using \cite[Lemma 4]{GV}, (\ref{poisson2}) and (\ref{pf19}), we find that
\begin{align}\label{pf20}
\begin{split}
\int_{-\infty}^{\infty} F(x)\ \dx &= \lim_{N\rightarrow\infty} \sum_{n=-N}^N\Bigl(1 - \frac{|n|}{N+1}\Bigr) F(n + v)\\
     &\le \lim_{N\rightarrow\infty} \sum_{n=-N}^N\Bigl(1 - \frac{|n|}{N+1}\Bigr) G_{\lambda}(n + v)\\
     &= \lambda^{-\h}\theta_2\bigr(\hh - v, i\lambda^{-1}\bigr)
\end{split}
\end{align}
for all real $v$.  We have already noted that $v\mapsto \theta_2\bigl(\hh - v, i\lambda^{-1}\bigr)$ takes its minimum value
at $v = \h$.  Hence (\ref{pf20}) implies that
\begin{equation*}\label{pf21}
\int_{-\infty}^{\infty} F(x)\ \dx \le \lambda^{-\h}\theta_2\bigr(0, i\lambda^{-1}\bigr),
\end{equation*}
and this proves (\ref{intro21}). 

In Corollary \ref{cor9} we proved that $F(z) = L_{\lambda}(z)$ satisfies the inequality (\ref{pf19}) for all real $x$.
In this special case there is equality in the inequality (\ref{pf20}) when $v = \hh$.  Thus we have
\begin{equation}\label{pf22}
\int_{-\infty}^{\infty} L_{\lambda}(x)\ \dx = \lambda^{-\h}\theta_2\bigr(0, i\lambda^{-1}\bigr).
\end{equation}

Now {\it assume} that $F(z)$ is an entire function of exponential type at most $2\pi$ that satisfies (\ref{pf19}) for all
real $x$, and assume that there is equality in the inequality (\ref{pf20}).  It follows that $v= \hh$ and
\begin{equation*}\label{pf23}
F(n + \hh) = G_{\lambda}(n + \hh)
\end{equation*}
for all integers $n$.  Then from (\ref{pf19}) we also get
\begin{equation*}\label{pf24}
F^{\prime}(n + \hh) = G_{\lambda}^{\prime}(n + \hh)
\end{equation*}
for all integers $n$.  Of course this shows that the entire function
\begin{equation}\label{pf25}
z\mapsto F(z) - L_{\lambda}(z)
\end{equation}
has exponential type at most $2\pi$, vanishes at each point of $\Z + \hh$, and its derivative also vanishes at each
point of $\Z + \hh$.  By a second application of \cite[Lemma 4]{GV} we conclude that the entire function (\ref{pf25})
is identically zero.  This proves Theorem 2, and Theorem 3 can be proved by the same sort of argument.

\section{Proof of Theorem \ref{thm4}}

The partial sums for the series (\ref{theta1}) defining $t\mapsto \theta_1(t, i\lambda)$ converge absolutely and uniformly for $t$ in $\R$.  Therefore
we find that
\begin{align}\label{pf30}
\begin{split}
\int_{-\h}^{\h} \theta_1(t, i\lambda) e(tz)\ \dt &= \sum_{n=-\infty}^{\infty} e^{-\pi\lambda (n+\h)^2} \int_{-\h}^{\h} e\bigl(t(z+n+\hh)\bigr)\ \dt\\
     &=  \Bigl(\frac{\cos \pi z}{\pi}\Bigr) \bigg\{\sum_{n=-\infty}^{\infty} (-1)^n \frac{e^{-\pi\lambda (n+\h)^2}}{\bigl(z+n+\h\bigr)}\bigg\}\\
     &= K_{\lambda}(z).
\end{split}
\end{align}
Then (\ref{intro32}) follows from (\ref{pf30}) by Fourier inversion.  In a similar manner we find that
\begin{align}\label{pf33}
\begin{split}
\int_{-1}^1 \bigl(1 - |t|\bigr)&\theta_1(t, i\lambda) e(tz)\ \dt\\
	&= \sum_{n=-\infty}^{\infty} e^{-\pi\lambda(n+\h)^2} \int_{-1}^1\bigl(1 - |t|\bigr) e\bigl(t(z+n+\hh)\bigr)\ \dt\\
	&= \sum_{-\infty}^{\infty} e^{-\pi\lambda(n+\h)^2} \left(\frac{\sin \pi(z+n+\hh)}{\pi(z+n+\hh)}\right)^2\\
	&= \left(\frac{\cos \pi z}{\pi}\right)^2 \sum_{m=-\infty}^{\infty} \frac{G_{\lambda}\bigl(m+\hh\bigr)}{\bigl(z-m-\h\bigr)^2},
\end{split}
\end{align}
and
\begin{align}\label{pf34}
\begin{split}
-(2\pi)^{-1} \lambda \int_{-1}^1&\sgn(t)\frac{\partial\theta_1}{\partial t}(t, i\lambda) e(tz)\ \dt\\
	&= -(2\pi)^{-1} \lambda \sum_{-\infty}^{\infty} e^{-\pi\lambda(n+\h)^2}\big\{2\pi i(n+\hh)\big\} \int_{-1}^1 \sgn(t) e(tz)\ \dt\\
	&= \lambda \sum_{n=-\infty}^{\infty} e^{-\pi\lambda(n+\h)^2} \left\{\frac{2\pi(n+\hh)}{z+n+\hh}\right\}\left(\frac{\sin \pi(z+n+\hh)}{\pi}\right)^2\\
	&= \left(\frac{\cos \pi z}{\pi}\right)^2 \sum_{n=-\infty}^{\infty} \frac{G_{\lambda}^{\prime}\bigl(n+\hh\bigr)}{\bigl(z-n-\hh\bigr)}.
\end{split}
\end{align}
Now (\ref{intro33}) follows from (\ref{pf33}) and (\ref{pf34}) by Fourier inversion.  The proof of (\ref{intro34}) is essentially the same.  We note that
(\ref{intro33}) and (\ref{intro34}) are both special cases of a general formula for Fourier transforms given in \cite[Theorem 9]{V}.

\section{Proof of Theorem \ref{thm5} and \ref{thm6}}

Let $\delta = 2N+2$, and let $p$ be a trigonometric polynomial of degree at most $N$. The Fourier expansions \eqref{theta1} and \eqref{pf4} imply
\begin{align}\label{p5eq1}
\begin{split}
\int_{\R/\Z} \bigl|\theta_3\bigl(x,i\lambda^{-1}\bigr) - p(x)\bigr|\, \dx &\ge \Big|\int_{\R/\Z} \{\theta_3(x,i\lambda^{-1})-p(x)\}\,\sgn(\cos\pi\delta x)\, \dx\Big| \\
& =\Big|\int_{\R/\Z} \theta_3\bigl(x,i\lambda^{-1}\bigr)\,\sgn(\cos\pi\delta x)\,\dx\Big|\\
&= \Big|\frac{2}{\pi} \sum_{n=-\infty}^\infty \frac{(-1)^n}{2n+1}\, e^{-\pi\lambda^{-1}(n+\frac{1}{2})^2\delta^2}\Big|\\
&=\int_{-\frac{1}{2}}^{\frac12} \theta_1\bigl(u,i\lambda^{-1}\delta^2\bigr)\,\du,
\end{split}
\end{align}
which proves \eqref{intro40}. The representation of $\theta_3\bigl(v,i\lambda^{-1}\bigr)$ in \eqref{poisson3} and the second representation of $k_{\lambda,N}$ in \eqref{intro36} imply
\begin{align*}
\theta_3\bigl(x,i\lambda^{-1}\bigr)-k_{\lambda,N}(x) = \lambda^{\h} \sum_{m=-\infty}^{\infty}\left\{ e^{-\pi \lambda (x+m)^2} -  K_{\delta^{-2}\lambda}\bigl(\delta (x+m)\bigr)\right\}
\end{align*}
for all $x\in\R/\Z$. The identity $G_{\delta^{-2}\lambda}(\delta x) = G_\lambda(x)$ and \eqref{theta8} give
\begin{align}\label{p5eq3}
\text{sgn}\left\{\theta_3\bigl(x,i\lambda^{-1}\bigr)-k_{\lambda,N}(x)\right\} =\text{sgn}(\cos\pi\delta x),
\end{align}
hence for $p=k_{\lambda,N}$ we have equality in \eqref{p5eq1}. To prove uniqueness, note that in order to have equality in \eqref{p5eq1} we must have
\begin{align}\label{p5eq4}
p\left(\frac{n+\frac12 }{\delta}\right) = \theta_3\left(\frac{n+\frac12}{\delta}, i\lambda^{-1}\right)
\end{align}
for $n=0,1,2,...,2N+1$. Since the degree of $p(x)$ is at most $N$, such polynomial exists and is unique \cite[Vol II, page 1]{Z}, and we showed in \eqref{p5eq3} that $k_{\lambda,N}$ already satisfies \eqref{p5eq4}. This completes the proof of Theorem \ref{thm5}

We now prove the minorant part of Theorem \ref{thm6}.  Let $\delta:= N+1$.  Let $q$ be a trigonometric polynomial of degree at most $N$ satisfying $q(x) \leq \theta_3\bigl(x,i\lambda^{-1}\bigr)$ at each $x\in \R\backslash\Z$. We use the fact that the degree of $q$ is at most $N$ (and $\delta=N+1$), and we apply \eqref{poisson2} to obtain
\begin{align}\label{th6eq1}
\int_{\R/\Z} q(x)\,\dx =\delta^{-1} \sum_{n=0}^N q\left(\frac{n+\frac12}{\delta}\right)\le \delta^{-1} \sum_{n=0}^N\theta_3\left(\frac{n+\frac12}{\delta},i\lambda^{-1}\right)= \theta_2\bigl(0,i\lambda^{-1}\delta^{2}\bigr),
\end{align}
which proves \eqref{intro45}. Equations \eqref{poisson3} and \eqref{intro41} give
\begin{align}\label{th6eq2}
\theta_3\bigl(x,i\lambda^{-1}\bigr)-l_{\lambda,N}(x) =\lambda^\h \sum_{m=-\infty}^\infty \left\{ e^{-\pi\lambda(x+m)^2}-L_{\delta^{-2}\lambda}\bigl(\delta(x+m)\bigr) \right\}\ge 0
\end{align}
for all $x\in\R/\Z$, which proves the first inequality in \eqref{intro43}. Since $L_\lambda$ interpolates the values of $G_\lambda$ at $\Z+\frac12$, we have equality in \eqref{th6eq2}, and hence
\begin{align}\label{th6eq3}
l_{\lambda,N}\left(\frac{n+\frac12}{\delta} \right) =\theta_3\left(\frac{n+\frac12}{\delta},i\lambda^{-1}\right)
\end{align}
for $n=0,1,...,N$. Moreover, \eqref{intro41} and \eqref{intro21} imply
\begin{align}\label{th6eq4}
\int_{\R/\Z} l_{\lambda,N}(x)\,\dx = \lambda^\h\int_{-\infty}^\infty L_{\delta^{-2}\lambda}(\delta x)\,\dx = \theta_2\bigl(0,i\lambda^{-1}\delta^2\bigr),
\end{align}
hence equality occurs in \eqref{intro45} for $q = l_{\lambda,N}$. It remains to show uniqueness. If equality occurs in \eqref{th6eq1} then
\begin{align*}
q\left(\frac{n+\frac12}{\delta} \right)=\theta_3\left(\frac{n+\frac12}{\delta},i\lambda^{-1}\right) = l_{\lambda,N}\left(\frac{n+\frac12}{\delta} \right)
\end{align*}
for $n=0,...,N$. Since $q$ and $l_{\lambda,N}$ both minorize $\theta_3\bigl(x,i\lambda^{-1}\bigr)$, their derivatives at the points $\delta^{-1}\bigl(n+\hh\bigr)$ where $n=0,...,N$ have to be equal. Hence $q$ satisfies $2N+2$ conditions which determine a unique trigonometric polynomial of degree at most $N$ \cite[Vol. II, p. 23]{Z}, and therefore $q=l_{\lambda,N}$. The proof for the extremal majorizing trigonometric polynomial proceeds on analogous lines with interpolation points at $\delta^{-1} n$ for $n=0,...,N$.

\section*{Part II: Distribution Approach}
\section{The Paley-Wiener Theorem for Distributions}
Let $\D(\R) \subseteq \sS(\R) \subseteq \E(\R)$ be the usual spaces of $C^{\infty}$ functions on $\R$ as defined
in the work of L. Schwartz \cite{Sch}, and let $\E^{\prime}(\R) \subseteq \sS^{\prime}(\R) \subseteq \D^{\prime}(\R)$
be the corresponding dual spaces of distributions.  Our notation and terminology for distributions follows that of \cite{Gf},
and precise definitions for these spaces are given in \cite[Section 2.3]{Gf}.
We write $\p(x)$ for a generic element in the space $\sS(\R)$ of Schwartz functions.  If
$T$ in $\sS^{\prime}(\R)$ is a tempered distribution we write $T(\p)$ for the value of $T$ at $\p$.  Then the
Fourier transform of $T$ is the tempered distribution $\tT$ defined by
\begin{equation*}\label{dist1}
\tT(\p) = T(\tp),
\end{equation*}
where
\begin{equation*}\label{dist2}
\tp(y) = \int_{-\infty}^{\infty} \p(x) e(-yx)\ \dx
\end{equation*}
is the Fourier transform of the function $\p$. Functions $g: \R \to \R$ in any $L^p$ class or with polynomial growth can be regarded as elements of $\sS^{\prime}(\R)$ and we will usually make the identification
\begin{equation*}
g(\p)  = \int_{-\infty}^{\infty} g(x) \, \varphi(x)\, \dx
\end{equation*}
for all $\p$ in $\sS(\R)$.

We recall the following form of the Paley-Wiener theorem for distributions, which is obtained by combining
Theorem 1.7.5 and Theorem 1.7.7 in \cite{Hor}.

\begin{theorem}[Paley-Wiener for distributions]\label{thm1dist}
Let $\delta > 0$, and let $U$ be a tempered distribution in $\sS^{\prime}(\R)$ with Fourier transform $\tU$ supported in the compact 
interval $[-\delta, \delta]$.  Then $\tU$ belongs to $\E^{\prime}(\R)$, and
\begin{equation*}\label{dist3}
z\mapsto F(z) = \tU_{\xi}\big(e(\xi z)\big)
\end{equation*}
defines an entire function of the complex variable $z = x + iy$ such that
\begin{equation}\label{dist4}
\bigl|F(z)\bigr| \ll_B \bigl(1 + |z|\bigr)^B \exp\{2\pi\delta|y|\}
\end{equation}
for some number $B \ge 0$ and all $z$ in $\C$.  Moreover, the entire function $F(z)$ satisfies the identity
\begin{equation*}\label{dist5}
U(\p) = \int_{-\infty}^{\infty} F(x)\p(x)\ \dx
\end{equation*}
for all $\p$ in $\sS(\R)$.

Conversely, suppose that $F(z)$ is an entire function of the complex variable $z$ that satisfies the inequality {\rm (\ref{dist4})} 
for some numbers $B \ge 0$ and $\delta > 0$.  Then there exists a tempered distribution $V$ in $\sS^{\prime}(\R)$ such 
that $\widehat{V}$ belongs to $\E^{\prime}(\R)$, $\tV$ is supported on the compact interval $[-\delta, \delta]$,
\begin{equation*}\label{dist6}
F(z) = \tV_{\xi}\big(e(\xi z)\big),
\end{equation*}
and
\begin{equation*}\label{dist7}
V(\p) = \int_{-\infty}^{\infty} F(x)\p(x)\ \dx
\end{equation*}
for all $\p$ in $\sS(\R)$.
\end{theorem}

Here we write $\tU_{\xi}$ to indicate that the distribution $\tU$ is acting on the function $\xi \mapsto \big(e(\xi z)\big)$.

\section{Optimal Integration}\label{OptimalIntegration}

Throughout Part II of the paper we let $\lambda$ be a parameter on the interval $I \subset \R$ and consider a family of real valued even functions $x \mapsto G(\lambda, x)$ satisfying the following properties, for each $\lambda \in I$,\\
\begin{itemize}
 \item[(i)] The function $x \mapsto G(\lambda, x)$ is continuous on $\R$ and differentiable on $\R/\{0\}$.\\

\item[(ii)] There exist constants $C = C(\lambda)>0$ and $\epsilon = \epsilon(\lambda) >0$ such that, for all $x \in \R$ and $t \in \R$, 
\begin{equation*}
 |G(\lambda, x)| \leq \frac{C}{(1 + |x|)^{1 + \epsilon}}   \ \ \ \ \textrm{and} \ \ \ \ \ |\widehat{G}(\lambda, t)| \leq \frac{C}{(1 + |t|)^{1 + \epsilon}}.
\end{equation*}
\item[(iii)] The Fourier transform $t \mapsto \widehat{G}(\lambda,t)$ is non-negative and radially non-increasing.\\

\end{itemize}

Depending on the type of problem one wants to treat (minorant, majorant or best approximation) we will require one additional hypothesis about the family $G(\lambda, x)$, for each $\lambda \in I$, \\

\begin{itemize}
\item[(iv)] (Minorant) There is a unique extremal minorant $z\mapsto L(\lambda,z)$ of exponential type $2\pi$ that interpolates the values of $G(\lambda,x)$ at $\Z + \h$.\\

 \item[(v)] (Majorant) There is a unique extremal majorant $z\mapsto M(\lambda,z)$ of exponential type $2\pi$ that interpolates the values of $G(\lambda,x)$ at $\Z$.\\

 \item[(vi)] (Best Approximation) There is a unique best approximation $z\mapsto K(\lambda,z)$ of exponential type $\pi$ that interpolates the values of $G(\lambda,x)$ at $\Z + \h$ and satisfies
\begin{equation*}
 \sgn (\cos \pi x) \,\{G(\lambda,x) - K(\lambda, x)\} \geq 0.
\end{equation*}
\end{itemize}
\vspace{0.3cm}

We will call $\{x\mapsto G(\lambda, x)\}_{\lambda\in I}$ a \textsl{minorant family} if it satisfies properties (i)-(iv) above. The notions of \textsl{majorant family} and \textsl{best approximation family} are defined analogously using (v) and (vi) instead of (iv), respectively.

Observe that hypotheses (iv), (v) and (vi) do not need to coexist. Indeed, each problem can be treated independently. Examples of families of functions satisfying the conditions (i) - (iv) listed above that we have in mind for potential applications in this paper are given in the table below (note that in these cases $\lambda \in (0,\infty)$), together with the minimal values of the corresponding integrals.
\vspace{0.4cm}

\begin{center}

\begin{tabular}{|Sc|Sc|Sc|Sc|}
 \hline 

$G(\lambda,x) $ & Minorant  & Majorant  & Best Approximation \\

\hline

$e^{-\lambda \pi x^2}$ & $\scriptstyle\sum_{n\neq0} (-1)^n \, \lambda^{-\hh}\, e^{-\tfrac{\pi n^2}{\lambda}}$& $\scriptstyle\sum_{n\neq0} \, \lambda^{-\hh}\, e^{-\tfrac{\pi n^2}{\lambda}}$& $\scriptstyle\sum_{n=-\infty}^{\infty} \frac{(-1)^n \lambda^{-\hh}\, e^{-\tfrac{\pi \bigl(n+\hh\bigr)^2}{\lambda}}}{\pi \bigl(n+\hh\bigr)}$ \\

\hline

$e^{-\lambda |x|}$ & $\tfrac{2}{\lambda} - \csch\bigl(\tfrac{\lambda}{2}\bigr)$ & $\coth\bigl(\tfrac{\lambda}{2}\bigr) - \tfrac{2}{\lambda} $&$\tfrac{2}{\lambda} - \tfrac{2}{\lambda} \sech \bigl(\tfrac{\lambda}{2}\bigr)$ \\

\hline

$\frac{2\lambda}{\lambda^2 + 4 \pi^2 x^2}$ & $\frac{2}{e^{\lambda} + 1}$& $\frac{2}{e^{\lambda} - 1}$ & $\scriptstyle\sum_{n=-\infty}^{\infty} (-1)^n\, \frac{e^{-\lambda \bigl| n + \hh\bigr|}}{\pi\bigl( n + \tfrac{1}{2} \bigr)}$\\

\hline

\end{tabular}

\end{center}

\begin{center}
Table 1: Solution to the Beurling-Selberg problem for some functions.
\end{center}
\vspace{0.5cm}

Our goal here is to be able to integrate the parameter $\lambda$ with respect to a suitable non-negative Borel measure on $I$ and obtain the solution to a different extremal problem. One could first guess that the class of suitable measures $\nu$ on $I$ would consist of those measures for which the function
\begin{equation*}
 g(x) = \int_I G(\lambda, x)\, \dnu
\end{equation*}
is well defined, and that this would be the function to be approximated. Such method was essentially carried on in \cite{CV2}, \cite{CV3} and \cite{GV}. It turns out that this condition is usually restrictive and in order to find the optimal minimal conditions to be imposed on the measure $\nu$ one must look at things on the Fourier transform side. 

Let us illustrate what this condition should be on the minorant case. Define the difference function
\begin{equation*}
  D(\lambda, x) := G(\lambda, x) - L(\lambda, x) \geq 0.
\end{equation*}
The minimal integral corresponds to 
\begin{equation*}
 \int_{-\infty}^{\infty} \{ G(\lambda, x) - L(\lambda, x)\} \, \dx = \widehat{D}(\lambda, 0).
\end{equation*}
If we succeed in our attempt to integrate the parameter $\lambda$, we will end up solving an extremal problem for which the value of the minimal integral is given by (and thus we want to impose the finiteness)
\begin{equation}\label{MinIntFinal}
 \int_{I} \int_{-\infty}^{\infty} \{ G(\lambda, x) - L(\lambda, x)\} \, \dx\,\dnu =  \int_{I} \widehat{D}(\lambda, 0) \,\dnu < \infty.
\end{equation}
We will show that, when looked at the right angle, this is also a sufficient condition. \\

Suppose $\nu$ is a non-negative Borel measure on $I$ satisfying (\ref{MinIntFinal}). Since
\begin{equation*}
 |\widehat{D}(\lambda, t)| \leq \widehat{D}(\lambda, 0)
\end{equation*}
for all $t \in \R$, we observe that the function
\begin{equation*}
 t \mapsto \int_{I} \widehat{D}(\lambda, t) \, \dnu
\end{equation*}
is well defined. In particular, from the classical Paley-Wiener theorem, the Fourier transform $t \mapsto \widehat{L}(\lambda,t)$ is supported on $[-1,1]$ and therefore
\begin{equation*}
\int_{I} \widehat{D}(\lambda, t) \, \dnu = \int_{I} \widehat{G}(\lambda, t) \, \dnu
\end{equation*}
for $|t| \geq 1$. We are now in position to state the three main results of Part II of the paper. In the following theorems we write
\begin{equation*}\label{dist20}
[\alpha, \beta]^c = (-\infty, \alpha) \cup (\beta, \infty)
\end{equation*}
for the complement in $\R$ of a closed interval $[\alpha, \beta]$.

\begin{theorem}[Distribution Method - Minorant]\label{MinDis}
Let $\{x \mapsto G(\lambda,x)\}_{\lambda\in I}$ be a minorant family and $\nu$ be a non-negative Borel measure on $I$ satisfying 
\begin{equation*}
 \int_{I} \int_{-\infty}^{\infty} \{ G(\lambda, x) - L(\lambda, x)\} \, \dx\,\dnu < \infty.
\end{equation*}
Let $g: \R \to \R$ be a function on $\mc{S}'(\R)$ that is continuous on $\R/\{0\}$, differentiable on $\R/\{0\}$, and such that 
\begin{equation*}
 \widehat{g}(\varphi) = \int_{-\infty}^{\infty}\left\{\int_{I} \widehat{G}(\lambda, t) \, \dnu \right\} \varphi (t) \, \dt
\end{equation*}
for all Schwartz functions $\varphi$ supported on $[-1,1]^c$. Then there exists a unique extremal minorant $l(z)$ of exponential type $2\pi$ for $g(x)$. The function $l(x)$ interpolates the values of $g(x)$ at $\Z + \h$ and satisfies
\begin{equation*}
\int_{-\infty}^{\infty} \{g(x) -l(x) \} \,\dx =  \int_{I} \int_{-\infty}^{\infty} \{ G(\lambda, x) - L(\lambda, x)\} \, \dx\,\dnu.
\end{equation*}
\end{theorem}

\begin{theorem}[Distribution Method - Majorant]\label{MajDis}
Let $\{x \mapsto G(\lambda,x)\}_{\lambda\in I}$ be a majorant family and $\nu$ be a non-negative Borel measure on $I$ satisfying  
\begin{equation*}
 \int_{I} \int_{-\infty}^{\infty} \{ M(\lambda, x) - G(\lambda, x)\} \, \dx\,\dnu  < \infty.
\end{equation*}
Let $g: \R \to \R$ be a function on $\mc{S}'(\R)$ that is continuous on $\R$, differentiable on $\R/\{0\}$, and such that 
\begin{equation*}
 \widehat{g}(\varphi) = \int_{-\infty}^{\infty}\left\{\int_{I} \widehat{G}(\lambda, t)\, \dnu \right\} \varphi (t) \, \dt
\end{equation*}
for all Schwartz functions $\varphi$ supported on $[-1,1]^c$. Then there exists a unique extremal majorant $m(z)$ of exponential type $2\pi$ for $g(x)$. The function $m(x)$ interpolates the values of $g(x)$ at $\Z$ and satisfies
\begin{equation*}
\int_{-\infty}^{\infty} \{m(x) -g(x) \} \,\dx =  \int_{I} \int_{-\infty}^{\infty} \{ M(\lambda, x) - G(\lambda, x)\} \, \dx\,\dnu.
\end{equation*}
\end{theorem}

\begin{theorem}[Distribution Method - Best Approximation]\label{BADis}
Let $\{x \mapsto G(\lambda,x)\}_{\lambda\in I}$ be a best approximation family and $\nu$ be a non-negative Borel measure on $I$ satisfying 
\begin{equation*}
 \int_{I} \int_{-\infty}^{\infty} \left|G(\lambda, x) - K(\lambda, x)\right| \, \dx\,\dnu  < \infty.
\end{equation*}
Let $g: \R \to \R$ be a function on $\mc{S}'(\R)$ that is continuous on $\R/\{0\}$, and such that 
\begin{equation*}
 \widehat{g}(\varphi) = \int_{-\infty}^{\infty}\left\{\int_{I} \widehat{G}(\lambda, t)\, \dnu \right\} \varphi (t) \, \dt
\end{equation*}
for all Schwartz functions $\varphi$ supported on $[-\h,\h]^c$. Then there exists a unique best approximation $k(z)$ of exponential type $\pi$ for $g(x)$. The function $k(x)$ interpolates the values of $g(x)$ at $\Z + \h$, satisfying
\begin{equation*}
 \sgn (\cos \pi x)\,\{g(x) - k(x)\} \geq 0
\end{equation*}
and
\begin{equation*}
\int_{-\infty}^{\infty} |g(x) -k(x)| \,\dx =  \int_{I} \int_{-\infty}^{\infty} \left|G(\lambda, x) - K(\lambda, x)\right| \, \dx\,\dnu.
\end{equation*}
\end{theorem}
\vspace{0.2cm}

Similar results can be stated for the problem of majorizing or minorizing by functions of exponential type $2\pi\delta$, or best approximating by functions of type $\pi \delta$. It is a matter of changing the interpolation points to $\delta \Z$ or $\delta (\Z + \h)$, and changing the support intervals to $[-\delta, \delta]^c$ in the case of minorants/majorants and to $[-\frac{\delta}{2}, \frac{\delta}{2}]^c$ in the case of best approximation. For simplicity, we will proceed in our exposition only with type $2\pi$ for minorants/majorants and type $\pi$ for the best approximation problem.

The condition 
\begin{equation*}
 \widehat{g}(\varphi) = \int_{-\infty}^{\infty}\left\{\int_{I} \widehat{G}(\lambda, t)\, \dnu \right\} \varphi (t) \, \dt
\end{equation*}
for all Schwartz functions $\varphi$ supported on $[-\delta,\delta]^c$, that appears on the statements of the theorems, basically says that the Fourier transform $\widehat{g}$, which in principle is only a tempered distribution, is actually given by a function
\begin{equation*}
 t \mapsto \int_{I} \widehat{G}(\lambda, t)\, \dnu
\end{equation*}
outside the interval $[-\delta,\delta]$. This is a typical behavior of functions with polynomial growth, that might have Fourier transform given by a singular part supported on the origin plus an additional component given by a function outside the origin (e.g. the Fourier transform of $-\log |x|$ is given by $\frac{1}{2|t|}$ outside the origin). It is clear in this context that the only piece of information relevant for the Beurling-Selberg extremal problem is the knowledge of the Fourier transform of the original function outside a compact interval.

Finally, we shall see that this method applied to the Gaussian family 
$$G(\lambda, x) = G_{\lambda}(x) = e^{-\pi \lambda x^2}$$ 
is very powerful, producing most of the previously known examples in the literature and a wide class of new ones. In particular, we will be able to arrive at families of functions like
\begin{equation*}
 \widetilde{G}(\alpha, x) = \log \left( \frac{x^2 + \alpha^2}{x^2 + 4}\right)\ \ {\rm or} \ \ \widetilde{G}(\alpha, x) = |x|^{\alpha}\,,
\end{equation*}
and one could think about integrating the new parameter $\alpha$ to produce other functions. Although these families do not satisfy the original requirements (i)-(iv) this is a perfectly reasonable argument, since by Fubini's theorem, an integral with respect to the parameter $\alpha$ will only produce a different measure $\nu$ for the original integration on the parameter $\lambda$ for the Gaussian. Therefore, there is no loss of generality in starting the procedure with a nice family of functions satisfying the regularity requirements (i)-(iv) and iterating the method as desired.

\section{Proofs of Theorems \ref{MinDis} and \ref{MajDis}}
Here we give a detailed proof of Theorem \ref{MinDis}. The proof of  Theorem \ref{MajDis} follows the same general method.

First we construct the extreme minorant.  Recall the difference function 
\begin{equation*}
  D(\lambda, x) = G(\lambda, x) - L(\lambda, x) \geq 0\,,
\end{equation*}
and for each $x \in \R$ define the function
\begin{equation}\label{def of d}
 d(x) := \int_{I} D(\lambda, x)\, \dnu \geq 0.
\end{equation}
In principle, the value of $d(x)$ could be $\infty$ at some points. Observe, however, that $d(x) \in L^1(\R)$, since
\begin{equation*}
 \int_{-\infty}^{\infty} d(x)\, \dx = \int_{I} \int_{-\infty}^{\infty} D(\lambda, x)\, \dx \, \dnu = \int_{I} \widehat{D}(\lambda, 0) \, \dnu < \infty\,,
\end{equation*}
by the hypotheses of our theorem. The Fourier transform $\widehat{d}(t)$ is thus a continuous function given by
\begin{align}
\begin{split} \label{FTd0}
 \widehat{d}(t) = \int_{-\infty}^{\infty} d(x) & \, e(-tx) \, \dx = \int_{-\infty}^{\infty}  \int_{I} D(\lambda, x)\,  e(-tx) \, \dnu \, \dx\\
& = \int_{I}  \int_{-\infty}^{\infty}  D(\lambda, x)\, e(-tx) \, \dx\, \dnu = \int_{I} \widehat{D}(\lambda, t)\, \dnu \,,
\end{split}
\end{align}
and observe that for $|t|\geq 1$ we have
\begin{equation}\label{FTd}
 \widehat{d}(t) = \int_{I} \widehat{G}(\lambda, t)\, \dnu.
\end{equation}

Let $U \in \mc{S}'(\R)$ be the tempered distribution given by
\begin{equation}\label{temp1}
 U(\varphi) = \int_{-\infty}^{\infty} \{g(x) - d(x)\}\, \varphi(x)\, \dx.
\end{equation}
We shall prove that the Fourier transform $\widehat{U}$ is supported on $[-1,1]$. In fact, for any $\varphi \in \mc{S}(\R)$ with support in $[-1,1]^c$ we have
\begin{align*}
 \widehat{U}(\varphi) & = \widehat{g} (\varphi) - \widehat{d}(\varphi) \\
& = \int_{-\infty}^{\infty}\left\{\int_{I} \widehat{G}(\lambda, t)\, \dnu \right\} \varphi (t) \, \dt - \int_{-\infty}^{\infty} \widehat{d}(t) \,\varphi(t)\, \dt = 0\,,
\end{align*}
by (\ref{FTd}) and the hypotheses of the theorem. By the Paley-Wiener theorem for distributions we find out that $\widehat{U} \in \mc{E}'(\R)$ and 
\begin{equation*}
 z \mapsto l(z) = \widehat{U}_{\xi} \left(e(\xi z)\right)
\end{equation*}
defines an entire function of exponential type $2\pi$ such that 
\begin{equation}\label{temp2}
 U(\varphi)  = \int_{-\infty}^{\infty} l(x)\, \varphi(x) \, \dx
\end{equation}
for all $\varphi \in \mc{S}(\R)$. From (\ref{temp1}) and (\ref{temp2}) we conclude that 
\begin{equation*}
d(x) = g(x) - l(x)
\end{equation*}
for almost all $x \in \R$. In particular, 
\begin{equation}\label{min8.1}
l(x) \leq g(x)
\end{equation}
for all $x \in \R$, and
\begin{align*}\label{MinIntDistributionApp}
\begin{split}
 \int_{-\infty}^{\infty} \{g(x) - l(x) \} \, \dx &= \int_{-\infty}^{\infty} d(x)\,\dx = \int_{I} \widehat{D}(\lambda, 0) \,\dnu \\
 & =  \int_{I} \int_{-\infty}^{\infty} \{ G(\lambda, x) - L(\lambda, x)\} \, \dx\,\dnu < \infty.
\end{split}
\end{align*}

Next we consider the interpolation points.  Because of conditions (i) and (ii), the Poisson summation formula can be applied to $D(\lambda,x)$ giving a pointwise identity
\begin{equation}\label{PSF8.1}
\sum_{n=-\infty}^{\infty} D(\lambda, x+n) = \sum_{k=-\infty}^{\infty} \widehat{D}(\lambda, k) \,e(xk).
\end{equation}
From condition (iv) of our hypotheses we have $D\bigl(\lambda, n + \h\bigr) = 0$ for all $n \in \Z$. Thus, plugging $x = \h$ in (\ref{PSF8.1}) and using the classical Paley-Wiener theorem we arrive at 
\begin{equation}\label{PSF8.2}
\widehat{D}(\lambda, 0)  = - \sum_{\stackrel{k=-\infty}{k \neq 0}}^{\infty} (-1)^k \,\widehat{G}(\lambda, k) .
\end{equation}
Now we define the function
\begin{equation*}
d_1(x) := g(x) - l(x)\,,
\end{equation*}
and observe that $d_1(x)$ is a non-negative continuous function on $\R/\{0\}$ that is equal almost everywhere to $d(x)$ defined in (\ref{def of d}), and thus in $L^1(\R)$. Define a periodic function $p: \R/\Z \to \R^{+} \cup \{\infty\}$ by
\begin{equation*}
p(x) : = \sum_{n \in \Z} d_1(n+x).
\end{equation*}
An application of Fubini's theorem provides
\begin{equation*}
 \int_{\R/\Z} p(x)\, \dx = \int_{-\infty}^{\infty} d_1(x)\, \dx < \infty\,,
\end{equation*}
and therefore $p(x) \in L^1(\R/\Z)$. Moreover, the Fourier coefficients of $p(x)$ satisfy
\begin{equation*}
 \widehat{p}(k) = \widehat{d}_1(k) = \widehat{d}(k)
\end{equation*}
for all $k \in \Z$. Convolution with the smoothing F\'{e}jer kernel
\begin{equation*}
 F_N (x) = \frac{1}{N+1} \left( \frac{\sin \pi(N+1) x}{\sin \pi x} \right)^2
\end{equation*}
produces the pointwise identity
\begin{align*}
p*F_N(x) &= \sum_{k=-N}^{N} \left(1 - \frac{|k|}{N}\right)\, \widehat{p}(k)\, e(xk)\\
&  = \widehat{d}(0) + \sum_{\stackrel{k=-N}{k \neq0}}^{N} \left(1 - \frac{|k|}{N}\right)\, \widehat{d}(k)\, e(xk)\\
 & = \widehat{d}(0) + \sum_{\stackrel{k=-N}{k \neq0}}^{N} \left(1 - \frac{|k|}{N}\right)\,\int_{I} \widehat{G}(\lambda, k)\, \dnu \, e(xk)\\
& = \widehat{d}(0) + \int_{I} \left\{\sum_{\stackrel{k=-N}{k \neq0}}^{N} \left(1 - \frac{|k|}{N}\right)\, \widehat{G}(\lambda, k) \, e(xk)\right\} \, \dnu\,,
\end{align*}
where we have used (\ref{FTd}). In particular, for $x = \h$ we obtain
\begin{equation}\label{Fatou0}
\widehat{d}(0) = p*F_N\bigl(\tfrac{1}{2}\bigr) + \int_{I}  \left\{\sum_{\stackrel{k=-N}{k \neq0}}^{N} (-1)^{k+1} \left(1 - \frac{|k|}{N}\right)\, \widehat{G}(\lambda, k) \right\} \, \dnu.
\end{equation}
By condition (iii) of the hypotheses, the integrand in (\ref{Fatou0}) in non-negative. Moreover, by condition (ii) it converges absolutely to (\ref{PSF8.2}) as $N \to \infty$. Therefore, an application of Fatou's lemma together with (\ref{FTd0}) gives us
\begin{align*}
\widehat{d}(0)  & \geq \liminf_{N\to \infty}  p*F_N\bigl(\tfrac{1}{2}\bigr) \\
& \ \ \ \ \ \ \ \ \ \ \ + \liminf_{N\to \infty}\int_{I}  \left\{\sum_{\stackrel{k=-N}{k \neq0}}^{N} (-1)^{k+1} \left(1 - \frac{|k|}{N}\right)\, \widehat{G}(\lambda, k) \right\} \, \dnu\\
& \geq \liminf_{N\to \infty}  p*F_N\bigl(\tfrac{1}{2}\bigr)\\
& \ \ \ \ \ \ \ \ \ \ \ + \int_{I} \liminf_{N\to \infty}\left\{\sum_{\stackrel{k=-N}{k \neq0}}^{N} (-1)^{k+1} \left(1 - \frac{|k|}{N}\right)\, \widehat{G}(\lambda, k) \right\} \, \dnu\\
&\\
& = \liminf_{N\to \infty}  p*F_N\bigl(\tfrac{1}{2}\bigr) + \int_{I} \widehat{D}(\lambda, 0) \, \dnu\\
&\\
& = \liminf_{N\to \infty}  p*F_N\bigl(\tfrac{1}{2}\bigr) + \widehat{d}(0)\,,
\end{align*}
and since $p*F_N(x)$ is non-negative we conclude that
\begin{equation*}\label{Fatou1}
 \liminf_{N\to \infty}  p*F_N\bigl(\tfrac{1}{2}\bigr) = 0.
\end{equation*}
We now use the definition of $p(x)$, Fubini's theorem and Fatou's lemma again to arrive at 
\begin{align}\label{Fatou2}
\begin{split}
 0 = \liminf_{N\to \infty}  & \,p*F_N\bigl(\tfrac{1}{2}\bigr) = \liminf_{N\to \infty} \int_{0}^{1} p(y)\,F_N\bigl(\tfrac{1}{2} - y\bigr)\,\dy\\
& = \liminf_{N\to \infty} \int_{0}^{1} \left\{\sum_{n \in \Z} d_1(n+y)\right\}\,F_N\bigl(\tfrac{1}{2} - y\bigr)\,\dy\\
& = \liminf_{N\to \infty} \sum_{n \in \Z} \left\{ \int_{0}^{1} d_1(n+y)\,F_N\bigl(\tfrac{1}{2} - y\bigr)\,\dy\right\}\\
& \geq \sum_{n \in \Z} \, \liminf_{N\to \infty} \, \int_{0}^{1} d_1(n+y)\,F_N\bigl(\tfrac{1}{2} - y\bigr)\,\dy\,\\
& = \sum_{n \in \Z} d_1\bigl(n + \tfrac{1}{2}\bigr)\,,
\end{split}
\end{align}
where the last equality follows from the fact that $d_1(x)$ is continuous at the points $n + \h$, $n \in \Z$. From (\ref{Fatou2}) and the non-negativity of $d_1(x)$ we arrive at the conclusion
\begin{equation}\label{interpolation8.1}
d_1\bigl(n + \tfrac{1}{2}\bigr) = 0 \Rightarrow g\bigl(n + \tfrac{1}{2}\bigr) = l\bigl(n + \tfrac{1}{2}\bigr)
\end{equation}
for all $n \in \Z$. From (\ref{min8.1}) and the fact that $g(x)$ is differentiable on $\R/\{0\}$ (by hypothesis) we also have
\begin{equation*}\label{interpolation8.2}
 g'\bigl(n + \tfrac{1}{2}\bigr) = l'\bigl(n + \tfrac{1}{2}\bigr)
\end{equation*}
for all $n \in \Z$.

Finally, we show that the integral is minimal and we establish uniqueness.  Assume that $\widetilde{l}(z)$ is a real entire function of exponential 
type $2\pi$ such that 
\begin{equation}\label{min8.2}
 \widetilde{l}(x) \leq g(x)
\end{equation}
for all $x \in \R$, and suppose that $\{g(x) - \widetilde{l}(x)\}$ is integrable. In this case the function
\begin{equation*}
j(z)  = l(z) - \widetilde{l}(z)
\end{equation*}
has exponential type $2\pi$ and is integrable on $\R$. An application of \cite[Lemma 4]{GV} together with (\ref{interpolation8.1}) and (\ref{min8.2}) gives us
\begin{align}\label{inter8.2}
 \begin{split}
\widehat{j}(0) &= \lim_{N\to \infty} \sum_{n = -N}^{N} \left(1 - \frac{|n|}{N} \right)\, j\bigl(n+ \tfrac{1}{2} \bigr)\\
& = \lim_{N\to \infty} \sum_{n = -N}^{N} \left(1 - \frac{|n|}{N} \right)\, \left( g\bigl(n + \tfrac{1}{2} \bigr) - \widetilde{l}\bigl(n + \tfrac{1}{2} \bigr) \right) \geq 0\,.
\end{split}
\end{align}
This plainly verifies that 
\begin{equation*}
 \int_{-\infty}^{\infty} \{g(x) - \widetilde{l}(x)\}\, \dx \geq \int_{-\infty}^{\infty} \{g(x) - l(x)\}\, \dx\,,
\end{equation*}
proving the minimality of the integral. If equality occurs in (\ref{inter8.2}) we must have 
\begin{equation}\label{interpolation8.3}
 \widetilde{l}\bigl(n + \tfrac{1}{2}\bigr) = g\bigl(n + \tfrac{1}{2} \bigr) = l\bigl(n + \tfrac{1}{2} \bigr) 
\end{equation}
for all $n \in \Z$. From (\ref{min8.2}) we also have
\begin{equation}\label{interpolation8.4}
 \widetilde{l}\,'\bigl(n + \tfrac{1}{2}\bigr) = g'\bigl(n + \tfrac{1}{2} \bigr) = l'\bigl(n + \tfrac{1}{2} \bigr)
\end{equation}
for all $n \in \Z$. The interpolation conditions (\ref{interpolation8.3}) and (\ref{interpolation8.4}) imply that
\begin{equation*}
 j\bigl(n + \tfrac{1}{2} \bigr) = j'\bigl(n + \tfrac{1}{2} \bigr) =0
\end{equation*}
for all $n \in \Z$. By a second application of \cite[Lemma 4]{GV}, we conclude that the entire function $j(z)$ must be identically zero, thus proving the uniqueness of the extremal minorant $l(z)$. This finishes the proof.

We note that in the proof of uniqueness in the majorant case, we will obtain 
\begin{equation*}
 j'(n) =0
\end{equation*}
for all $n \neq 0$, since the original function $g(x)$ is not supposed to be differentiable at the origin.  A further application 
of \cite[Lemma 4]{GV} provides $j'(0) = 0$, thus leading to uniqueness.

\section{Proof of Theorem \ref{BADis}}

The approach here is similar to the proof of Theorem \ref{MajDis}.  We start by considering the difference function
\begin{equation*}
  D(\lambda, x) = G(\lambda, x) - K(\lambda, x)\,,
\end{equation*}
and for each $x \in \R$ define the function
\begin{equation*}
 d(x) = \int_{I} D(\lambda, x)\, \dnu\,.
\end{equation*}
From condition (v) on the hypotheses, we know that
\begin{equation}\label{sgnd}
 \sgn(\cos \pi x) \,d(x) \geq 0\,,
\end{equation}
with this value being possibly infinite at some points. Observe, however, that $d(x)$ is integrable on $\R$, with
\begin{align*}
\int_{-\infty}^{\infty} |d(x)|\, \dx &= \int_{-\infty}^{\infty} \int_{I} |D(\lambda, x)|\, \dnu \,\dx \\
&  =  \int_{I}  \int_{-\infty}^{\infty} |G(\lambda, x) - K(\lambda,x)|\, \dx\, \dnu < \infty.
\end{align*}
An application of Fubini's theorem gives us 
\begin{equation*}
 \widehat{d}(t) = \int_{I} \widehat{D}(\lambda, t)\, \dnu
\end{equation*}
for all $t \in \R$, and since $z \mapsto K(\lambda, z)$ has exponential type $\pi$, we have
\begin{equation}\label{BA1}
 \widehat{d}(t) = \int_{I} \widehat{G}(\lambda, t)\, \dnu
\end{equation}
for $|t|\geq \h$. Let $V \in \mc{S}'(\R)$ be the tempered distribution given by
\begin{equation}\label{temp9.1}
 V(\varphi) = \int_{-\infty}^{\infty} \{g(x) - d(x)\}\, \varphi(x)\, \dx.
\end{equation}
We shall prove that the Fourier transform $\widehat{V}$ is supported on $\bigl[-\h,\h\bigr]$. In fact, for any $\varphi \in \mc{S}(\R)$ with support in $\bigl[-\h,\h\bigr]^c$ we have
\begin{align*}
 \widehat{V}(\varphi) & = \widehat{g} (\varphi) - \widehat{d}(\varphi) \\
& = \int_{-\infty}^{\infty}\left\{\int_{I} \widehat{G}(\lambda, t)\, \dnu \right\} \varphi (t) \, \dt - \int_{-\infty}^{\infty} \widehat{d}(t) \,\varphi(t)\, \dt = 0\,,
\end{align*}
by (\ref{BA1}) and the hypotheses of the theorem. By the Paley-Wiener theorem for distributions we find out that $\widehat{V} \in \mc{E}'(\R)$ and 
\begin{equation*}
 z \mapsto k(z) = \widehat{V}_{\xi} \left(e(\xi z)\right)
\end{equation*}
defines an entire function of exponential type $\pi$ such that 
\begin{equation}\label{temp9.2}
 V(\varphi)  = \int_{-\infty}^{\infty} k(x)\, \varphi(x) \, \dx
\end{equation}
for all $\varphi \in \mc{S}(\R)$. From (\ref{temp9.1}) and (\ref{temp9.2}) we conclude that 
\begin{equation}\label{sgnd2}
d(x) = g(x) - k(x)
\end{equation}
for almost all $x \in \R$. In particular,
\begin{align*}
 \int_{-\infty}^{\infty} |g(x) - k(x)|\, \dx & = \int_{-\infty}^{\infty} |d(x)|\, \dx \\
& = \int_{I}  \int_{-\infty}^{\infty} |G(\lambda, x) - K(\lambda,x)|\, \dx\, \dnu < \infty\,.
\end{align*}

Since $g(x)$ is continuous on $\R/\{0\}$ (by hypothesis) and $k(x)$ is the restriction to $\R$ of an entire function, expressions 
(\ref{sgnd}) and (\ref{sgnd2}) imply that 
\begin{equation*}
\sgn(\cos \pi x) \,\{g(x) - k(x)\} \geq 0
\end{equation*}
for all $x \in \R$. In particular, we must have 
\begin{equation*}
g\bigl( n + \tfrac{1}{2} \bigr) = k\bigl( n + \tfrac{1}{2} \bigr)
\end{equation*}
for all $n \in \Z$.

Recall that the function $x \mapsto \sgn(\cos \pi x)$ is periodic on $\R$ with period $2$ and has Fourier series expansion 
\begin{equation}\label{Sec9.3}
\sgn(\cos \pi x) = \lim_{N\rightarrow \infty} \frac{1}{\pi} \sum_{n=-N-1}^{N} \frac{(-1)^n}{n+\hh}\, e\bigl((n+\hh)x\bigr).
\end{equation}
Moreover, the partial sums on the right of (\ref{Sec9.3}) are uniformly bounded. If $\psi(x)$ is a function of exponential type $\pi$ that is integrable on $\R$, its Fourier transform $\widehat{\psi}(t)$ will be supported on $\bigl[ -\h, \h \bigr]$ and we will have
\begin{align}\label{Sec9.4}
\begin{split}
\int_{-\infty}^{\infty} \sgn(\cos & \pi x)\,\psi(x) \,\dx\\
&= \lim_{N\rightarrow\infty} \frac{1}{\pi} \sum_{n = -N-1}^{N} 
		\frac{(-1)^n}{n+\hh} \int_{-\infty}^{\infty} \psi(x)\, e\bigl((n+\hh)x\bigr)\ \dx\\
	&= \lim_{N\rightarrow\infty} \frac{1}{\pi} \sum_{n = -N-1}^{N}  \frac{(-1)^n}{n+\hh} \, \widehat{\psi}\bigl(-n - \hh \bigr) = 0.
\end{split}
\end{align}
Now assume that $\widetilde{k}(z)$ is an entire function of exponential type $\pi$ such that 
\begin{equation*}
\int_{-\infty}^{\infty} |g(x) - \widetilde{k}(x)|\, \dx < \infty.
\end{equation*}
In this case, the function $\{k(x) - \widetilde{k}(x)\}$ has exponential type $\pi$ and is integrable on $\R$. Thus, using (\ref{Sec9.4}) we obtain
\begin{align}\label{Sec9.5}
\begin{split}
\int_{-\infty}^{\infty}  |g(x) - &\widetilde{k}(x)|\, \dx  \geq \left| \int_{-\infty}^{\infty} \sgn(\cos \pi x) \,\{g(x) - \widetilde{k}(x)\}\, \dx \right|\\
& = \left| \int_{-\infty}^{\infty} \sgn(\cos \pi x) \,\bigl\{\bigl(g(x) - k(x)\bigr) + \bigl(k(x) - \widetilde{k}(x)\bigr)\bigr\} \, \dx \right|\\
&=\left| \int_{-\infty}^{\infty} \sgn(\cos \pi x) \,\{g(x) - k(x)\}\, \dx \right|\\
& = \int_{-\infty}^{\infty}  |g(x) - k(x)|\, \dx\,,
\end{split}
\end{align}
proving the minimality of the integral. If equality occurs in (\ref{Sec9.5}) we must have 
\begin{equation*}
\widetilde{k}\bigl( n + \tfrac{1}{2} \bigr) = g\bigl( n + \tfrac{1}{2} \bigr) =  k\bigl( n + \tfrac{1}{2} \bigr)
\end{equation*}
for all $n \in \Z$. Therefore
\begin{equation}\label{Sec9.6}
z\mapsto k(z) - \widetilde{k}(z)
\end{equation}
is an entire function of exponential type at most $\pi$ and takes the value zero at each point of the set $\Z + \hh$.  From basic interpolation theorems for entire functions of exponential type (see \cite[Vol. II, p. 275]{Z}), we conclude that the entire function (\ref{Sec9.6}) is identically zero, proving the uniqueness. This completes the proof.

\section*{Part III: Applications}

By combining the results from Part I and Part II, we are able to solve the Beurling-Selberg extremal problem for a wide class of even functions, extending the 
works \cite{CV2}, \cite{CV3}, \cite{GV} and \cite{Lit}. As mentioned in the Introduction, some the $L^1(\R)$-approximations (without the one-sided conditions) recover results of Sz.- Nagy \cite{nagy,Shapiro}. Nagy's results are applicable to functions with a Fourier transform that satisfies certain monotonicity conditions for $t\geq \delta>0$ and is either even or odd. It is an interesting open problem whether the extremals for all such functions can be obtained with our methods. Throughout Part III of this paper we consider extremal minorants/majorants of exponential type $2\pi$ and best approximations of exponential type $\pi$, unless otherwise specified.

\section{Positive Definite Functions}

Recall that in Part I we worked with the Gaussian family
\begin{equation*}
 G_{\lambda}(x) = e^{-\pi \lambda x^2}\,,
\end{equation*}
where $\lambda >0$ is a parameter, with Fourier transform $t \mapsto \widehat{G}_{\lambda}(t)$ given by
\begin{equation*}
\widehat{G}_{\lambda}(t) = \lambda^{-\hh} e^{-\pi \lambda^{-1} t^2}.
\end{equation*}

In Theorems \ref{thm1}, \ref{thm2} and \ref{thm3} we constructed, for each $\lambda >0$, the extremal minorant $L_{\lambda}(z)$, the extremal majorant $M_{\lambda}(z)$, and the best approximation $K_{\lambda}(z)$ for $G_{\lambda}(x)$. These functions satisfy all the hypotheses (i)-(vi) of the distribution method with the values of the minimal integrals given by
\begin{align}
\begin{split}
\int_{-\infty}^{\infty} \{G_{\lambda}(x) & - L_{\lambda}(x)\} \, \dx \\
 & = \lambda^{-\h} \Bigl(1 -\theta_2\bigl(0, i\lambda^{-1}\bigr)\Bigr) =    \sum_{\stackrel{n=-\infty}{n\neq0}}^{\infty} (-1)^n \widehat{G}_{\lambda}(n)\,,\label{Sec10.1} 
\end{split}\\
\begin{split}
\int_{-\infty}^{\infty} \{M_{\lambda}(x) & - G_{\lambda}(x)\} \ \dx \\ \label{Sec10.2}
&  = \lambda^{-\h} \Bigl(\theta_3\bigl(0, i\lambda^{-1}\bigr) -1 \Bigr)  = \sum_{\stackrel{n=-\infty}{n\neq0}}^{\infty} \widehat{G}_{\lambda}(n)\,,
\end{split}\\
\begin{split}
\int_{-\infty}^{\infty} \bigl|G_{\lambda}(x) &- K_{\lambda}(x)\bigr| \ \dx \\ \label{Sec10.3}
& = \lambda^{-\h} \int_{-\h}^{\h} \theta_1\bigl(u,i\lambda^{-1}\bigr)\ \du = \frac{1}{\pi} \sum_{n=-\infty}^{\infty} \frac{(-1)^n}{n+\hh} \,\tG_{\lambda}(n + \hh)\,.
\end{split}
\end{align}
From the three expressions above and the transformation formulas (\ref{poisson1}), (\ref{poisson2}) and (\ref{poisson3}) we obtain the following estimates 
\begin{equation}\label{Sec10.4}
 \int_{-\infty}^{\infty} \{G_{\lambda}(x)  - L_{\lambda}(x)\} \, \dx = O\bigl(e^{-\tfrac{\pi}{\lambda}}\bigr) \ \text{as} \ \lambda \to 0\,, \ \text{and}\ O\bigl(\lambda^{-\h}\bigr)\ \text{as} \ \lambda \to \infty\,,
\end{equation}
\begin{equation}\label{Sec10.5}
 \int_{-\infty}^{\infty} \{M_{\lambda}(x)  - G_{\lambda}(x)\} \, \dx = O\bigl(e^{-\tfrac{\pi}{\lambda}}\bigr) \ \text{as} \ \lambda \to 0\,, \ \text{and}\ O(1)\ \text{as} \ \ \lambda \to \infty\,,
\end{equation}
\begin{equation}\label{Sec10.6}
 \int_{-\infty}^{\infty} \bigl|G_{\lambda}(x) - K_{\lambda}(x)\bigr| \ \dx  = O\bigl(e^{-\tfrac{\pi}{4\lambda}}\bigr) \ \text{as} \ \lambda \to 0\,, \ \text{and}\ O\bigl(\lambda^{-\h}\bigr)\  \text{as} \  \lambda \to \infty\,.
\end{equation}
In order to apply Theorems \ref{MinDis}, \ref{MajDis}, and \ref{BADis}, to the Gaussian family, we require a non-negative measure $\nu$ defined
on the Borel subsets of $I = (0,\infty)$.  We further require that integrals with respect to $\nu$ over the parameter $\lambda$ appearing in (\ref{Sec10.1}), (\ref{Sec10.2}) and (\ref{Sec10.3}) be finite. The estimates (\ref{Sec10.4}), (\ref{Sec10.5}) and (\ref{Sec10.6}) show that this class of measures is wide because of the very fast decay at the origin. One should compare this class of measures with the ones used in \cite{CV2}, \cite{CV3}, and \cite{GV}, to fully notice the improvement.

As a first application we present the following result.

\begin{corollary}\label{thmPDF}
Let $\nu$ be a finite non-negative Borel measure on $(0,\infty)$ and consider the function $g: \R \to \R$ given by
\begin{equation}\label{posdef}
 g(x) = \int_0^{\infty} e^{-\pi \lambda x^2} \dnu\,.
\end{equation}
\begin{enumerate}
 \item[(i)] There exists a unique extremal minorant $l(z)$ of exponential type $2\pi$ for $g(x)$. The function $l(x)$ interpolates the values of $g(x)$ at $\Z + \h$ and satisfies
\begin{equation*}
 \int_{-\infty}^{\infty} \{g(x) - l(x) \}\, \dx = \int_{0}^{\infty} \left\{ \sum_{\stackrel{n=-\infty}{n\neq0}}^{\infty} (-1)^n \widehat{G}_{\lambda}(n)\right\} \, \dnu.
\end{equation*}

 \item[(ii)] There exists a unique extremal majorant $m(z)$ of exponential type $2\pi$ for $g(x)$. The function $m(x)$ interpolates the values of $g(x)$ at $\Z$ and satisfies
\begin{equation*}
 \int_{-\infty}^{\infty} \{m(x) - g(x) \}\, \dx = \int_{0}^{\infty} \left\{ \sum_{\stackrel{n=-\infty}{n\neq0}}^{\infty} \widehat{G}_{\lambda}(n)\right\} \, \dnu.
\end{equation*}

 \item[(iii)] There exists a unique best approximation $k(z)$ of exponential type $\pi$ for $g(x)$. The function $k(x)$ interpolates the values of $g(x)$ at $\Z + \h$ and satisfies
\begin{equation*}
 \int_{-\infty}^{\infty} |g(x) - k(x)|\, \dx = \int_{0}^{\infty} \left\{ \frac{1}{\pi} \sum_{n=-\infty}^{\infty} \frac{(-1)^n}{n+\hh} \,\tG_{\lambda}(n + \hh)\right\} \, \dnu.
\end{equation*}

\end{enumerate}

\end{corollary}

Due to a classical result of Schoenberg (see \cite[Theorems 2 and 3]{Scho}), a function $g:\R \to \R$ admits the representation  (\ref{posdef}) if and only if its radial extension to $\R^n$ is positive definite, for all $n \in \N$, or equivalently if the function $g\bigl(|x|^{1/2}\bigr)$ is completely monotone. Recall that a function $f(t)$ is {\it completely monotone} for $t\geq 0 $ if
\begin{equation*}
(-1)^n f^{(n)}(t) \geq0 \ \ {\rm for}  \ \ 0 < t < \infty\,,
\end{equation*}
and
\begin{equation*}
f(0) = f\bigl(0^{+}\bigr)\,,
\end{equation*}
the last condition expressing the continuity of $f(t)$ at the origin. Using this characterization we arrive at the following interesting examples contemplated by our Corollary \ref{thmPDF}.\\

{\it Example 1}.  \ \ $g(x) = e^{-\alpha|x|^{2r}}, \ \ \alpha\geq0 \ \ {\rm and} \ \ 0\leq r \leq 1.$\\

{\it Example 2}. \ \ $g(x) = \bigl(|x|^2 + \alpha^2\bigr)^{-\beta}\,, \ \ \alpha >0 \ \ \rm{and} \ \ \beta \geq 0.$\\

The first example shows in particular that we can recover all the theory for the exponential function $g(x) = e^{-\lambda |x|}$ developed in \cite{GV}, \cite{CV2} and \cite{CV3} from the Gaussian and the distribution method. The second example includes the Poisson kernel $g(x) = 2\lambda/ (\lambda^2 + 4 \pi^2 x^2)$, $\lambda >0$. The values of the minimal integrals in these cases are collected in the Table 1 of Section \ref{OptimalIntegration}.

Recently, Chandee and Soundararajan in \cite{CS} used the extremals for $f(x) = \log\bigl(x^2/(x^2 + 4)\bigr)$, described in \cite{CV2}, to obtain improved upper bounds for $|\zeta(\tfrac{1}{2} + it)|$ assuming the Riemann Hypothesis (RH). They remarked that the extremals for the function $f(x) = \log\bigl((x^2 + \alpha^2)/(x^2 + 4)\bigr)$, for $\alpha \neq 0$, not contemplated in the previous literature, arise in bounding $|\zeta(\tfrac{1}{2} \pm \alpha + it)|$, assuming RH, and might lead to improved bounds in the critical strip for inequalities of the type
\begin{equation}\label{Titchineq}
 \log \zeta(s) = O\left\{ \frac{(\log t)^{1-2\alpha}}{\log \log t}\right\}
\end{equation}
where $s = (1/2 + \alpha) + it$ and $0 \leq \alpha < 1/2$. Inequality (\ref{Titchineq}) can be found in \cite[Theorem 14.5]{Titch}.

Here we are able to obtain the extremals for this class of functions as an application of Corollary \ref{thmPDF}. \\
\\
{\it Example 3}. \ \ $g(x) = -\log\left(\frac{x^2 + \alpha^2}{x^2 + \beta ^2}\right) $, for $0 < \alpha < \beta$.\\
\\
Indeed, for $0 < \alpha < \beta$ consider the non-negative finite measure
\begin{equation*}
\dnu = \frac{\left\{e^{-\pi \lambda \alpha^2} - e^{-\pi \lambda \beta^2} \right\}}{\lambda}\, \dl\,,
\end{equation*}
and observe that 
\begin{equation*}
 -\log\left(\frac{x^2 + \alpha^2}{x^2 + \beta ^2}\right) = \int_0^{\infty}  e^{-\pi \lambda x^2} \frac{\left\{e^{-\pi \lambda \alpha^2} - e^{-\pi \lambda \beta^2} \right\}}{\lambda}\, \dl\,.
\end{equation*}
In particular, the values of the minimal integrals in the one-sided approximations are given by
\begin{equation*}
\int_{-\infty}^{\infty} \left\{-\log\left(\frac{x^2 + \alpha^2}{x^2 + \beta ^2}\right)- l_{\alpha,\beta}(x)\right\} \, \dx  =  2 \log\left(\frac{1 + e^{-2\pi\alpha}}{1 + e^{-2\pi\beta}}\right)\,,
\end{equation*}
and 
\begin{equation*}
\int_{-\infty}^{\infty} \left\{m_{\alpha,\beta}(x) + \log\left(\frac{x^2 + \alpha^2}{x^2 + \beta ^2}\right)\right\} \, \dx  =  2 \log\left(\frac{1 - e^{-2\pi\alpha}}{1 - e^{-2\pi\beta}}\right)\,.
\end{equation*}
We expect to return to the applications of these extremal functions to the theory of the Riemann zeta-function in a future work.

\section{Extremal Functions for $|x|^{\sigma}$}

Next we write $s = \sigma + it$ for a complex variable, and we define the meromorphic function $s\mapsto\gamma(s)$ by
\begin{equation*}\label{dist10}
\gamma(s) = \pi^{-s/2} \Gamma\left(\frac{s}{2}\right).
\end{equation*}
The function $\gamma(s)$ is analytic on $\C$ except for simple poles at the points $s = 0, -2, -4, \dots $.  It also occurs in the
functional equation
\begin{equation}\label{dist11}
\gamma(s)\zeta(s) = \gamma(1 - s) \zeta(1 - s),
\end{equation} 
where $\zeta(s)$ is the Riemann zeta-function.

\begin{lemma}\label{distlem2}
Let $0 < \delta$ and let $\p(t)$ be a Schwartz function supported on $[-\delta, \delta]^c$.  Then 
\begin{equation}\label{dist21}
s\mapsto \int_{-\infty}^{\infty} |t|^{-s-1} \p(t)\ \dt
\end{equation}
defines an entire function of $s$, and the identity
\begin{equation}\label{dist22}
\gamma(s+1) \int_{-\infty}^{\infty} |t|^{-s-1} \p(t)\ \dt = \gamma(-s) \int_{-\infty}^{\infty} |x|^{s}\, \tp(x)\ \dx
\end{equation}
holds in the half plane $\{s\in\C: -1< \sigma\}$.  In particular, the function on the right of {\rm (\ref{dist22})} is analytic at
the points $s = 0, 2, 4, \dots $.
\end{lemma}

\begin{proof}
Because $\p(t)$ is supported in $[-\delta, \delta]^c$, the function $t\mapsto |t|^{-s-1}\p(t)$ is integrable on $\R$ for
all complex values $s$.  Hence by Morera's theorem the integral on the right of (\ref{dist21}) defines an entire function.
The identity (\ref{dist22}) holds in the infinite strip $\{s\in\C: -1 < \sigma < 0\}$ by \cite[Lemma 1, p. 117]{Stein}, and therefore it
holds in the half plane $\{s\in\C: -1< \sigma \}$ by analytic continuation.  The left hand side of (\ref{dist22}) is clearly
analytic at each point of $\{s\in\C: -1 < \sigma \}$, hence the right hand side of (\ref{dist22}) is also analytic at
each point of this half plane. 
\end{proof}
Lemma \ref{distlem2} plainly says that the Fourier transform of the function $x \mapsto \gamma(-\sigma) |x|^{\sigma}$ is given by the function
\begin{equation*}
t \mapsto \gamma(\sigma+1) |t|^{-\sigma-1}
\end{equation*}
outside the interval $[-\delta, \delta]$, for $-1 < \sigma$, $\sigma \neq 0,2,4,...$. 

We intend to apply the distribution method with the Gaussian. For this, consider the non-negative Borel measure $\nu_{\sigma}$ on $(0,\infty)$ given by
\begin{equation*}
 \dnus = \lambda^{-\tfrac{\sigma}{2}-1}\, \dl\,,
\end{equation*}
and observe that we have exactly
\begin{equation}\label{sigmameasure}
 \int_0^{\infty} \widehat{G}_{\lambda}(t) \,\dnus = \gamma(\sigma+1) |t|^{-\sigma-1}.
\end{equation}
For $-1 < \sigma$, the measure $\nu_{\sigma}$ is admissible for the minorant and best approximation problems according to the asymptotics (\ref{Sec10.4}) and (\ref{Sec10.6}). For the majorant problem we shall require that $0 < \sigma$, according to the asymptotics (\ref{Sec10.5}). 

It will be convenient to introduce the Dirichlet $L$-function
$L(s, \chi)$, where $\chi$ is the unique nonprincipal Dirichlet character to the modulus $4$.  This $L$-function is
defined in the half plane $\{s\in \C: 1 < \sigma\}$ by the absolutely convergent series
\begin{equation*}\label{app50}
L(s, \chi) = \sum_{n=1}^{\infty} \chi(n) n^{-s} = \sum_{n=0}^{\infty} (-1)^n (2n+1)^{-s}.
\end{equation*}
Then the $L$-function extends by analytic continuation to an entire function of $s$.  As $\chi$ is a primitive
character, the $L$-function satisfies the functional equation
\begin{equation*}\label{app51}
\xi(s, \chi) = \xi(1 - s, \chi),
\end{equation*}
where $s\mapsto \xi(s, \chi)$ is the entire function defined by
\begin{equation}\label{app52}
\xi(s, \chi) = \left(\frac{4}{\pi}\right)^{\frac{s+1}{2}} \Gamma\left(\frac{s+1}{2}\right) L(s, \chi).
\end{equation}

\begin{lemma}\label{applem5}
Let $\sigma > -1$.  Then we have
\begin{equation}\label{app53}
\int_{-\infty}^{\infty} \int_0^{\infty} \bigl|G_{\lambda}(x) - K_{\lambda}(x)\bigr| \lambda^{-\frac{\sigma}{2} - 1}\ \dl\ \dx
	= \left(\frac{4}{\pi}\right)^{\frac{3+\sigma}{2}} \Gamma\left(\frac{1 +\sigma}{2}\right) L(2 +\sigma, \chi).
\end{equation}
\end{lemma}

\begin{proof}
Using Fubini's theorem, (\ref{theta1}), and (\ref{intro10}), we get
\begin{align}\label{app54}
\begin{split}
\int_{-\infty}^{\infty} \int_0^{\infty} \bigl|G_{\lambda}(x) - &K_{\lambda}(x)\bigr| \lambda^{-\frac{\sigma}{2} - 1}\ \dl\ \dx\\
	&= \int_0^{\infty} \left\{\lambda^{-\h} \int_{-\h}^{\h} \theta_1\bigl(u,i\lambda^{-1}\bigr)\ \du\right\} \lambda^{-\frac{\sigma}{2} - 1}\ \dl\\
	&= \int_0^{\infty} \left\{\sum_{n=-\infty}^{\infty} \frac{(-1)^n}{\pi(n+\hh)}\, e^{-\pi \lambda^{-1}(n+\h)^2} \right\} \lambda^{\frac{-\sigma - 3}{2}}\ \dl.
\end{split}
\end{align}
Because
\begin{equation*}\label{app55}
\begin{split}
\int_0^{\infty} \Bigg\{\sum_{n=-\infty}^{\infty}&\frac{1}{\pi|n+\hh|} e^{-\pi \lambda^{-1}(n+\h)^2} \Bigg\} \lambda^{\frac{-\sigma - 3}{2}}\ \dl\\
	&= \sum_{n=-\infty}^{\infty} \frac{1}{\pi|n+\hh|} \left\{\int_0^{\infty} \lambda^{\frac{-\sigma - 3}{2}} e^{-\pi \lambda^{-1} (n+\h)^2}\ \dl\right\}\\
	&= \gamma(1 + \sigma) \sum_{n=-\infty}^{\infty} \frac{1}{\pi} \bigl| n+\hh \bigr|^{-\sigma - 2} < \infty,
\end{split}
\end{equation*}
the partial sums of the series on the right of (\ref{app54}) are dominated by an integrable function.  Thus we have
\begin{align}\label{app56}
\begin{split}
\int_0^{\infty} \Bigg\{\sum_{n=-\infty}^{\infty} &\frac{(-1)^n}{\pi(n+\hh)} e^{-\pi \lambda^{-1}(n+\h)^2} \Bigg\} \lambda^{\frac{-\sigma - 3}{2}}\ \dl\\
	&=  \sum_{n=-\infty}^{\infty} \frac{(-1)^n}{\pi(n+\hh)} \left\{\int_0^{\infty} \lambda^{\frac{-\sigma - 3}{2}} e^{-\pi \lambda^{-1} (n+\h)^2}\ \dl\right\}\\
	&= \gamma(1 + \sigma) \sum_{n=-\infty}^{\infty} \frac{(-1)^n}{\pi(n+\hh)} \bigl| n+\hh \bigr|^{-\sigma - 1}\\
	&= \left(\frac{4}{\pi}\right)^{\frac{3+\sigma}{2}} \Gamma\left(\frac{1 + \sigma}{2}\right) L(2 + \sigma, \chi).
\end{split}
\end{align}
Identities (\ref{app54}) and (\ref{app56}) imply that the identity (\ref{app53}) holds for $\sigma >-1$.  
\end{proof}

The following lemma can be proved in a similar manner using Theorems \ref{thm2} and \ref{thm3}, and then
applying termwise integration to the series (\ref{theta2}) and (\ref{theta3}).	
	
\begin{lemma}\label{applem6}	
Let $\sigma > -1$.  Then we have
\begin{equation}\label{app64}	
\int_{-\infty}^{\infty} \int_0^{\infty} \big\{G_{\lambda}(x) - L_{\lambda}(x)\big\} \lambda^{-\frac{\sigma}{2} - 1}\ \dl\ \dx
	= \bigl(2 - 2^{1-\sigma}\bigr)\, \gamma(1 + \sigma)\,\zeta(1 + \sigma).
\end{equation}
Let $\sigma > 0$.  Then we have
\begin{equation}\label{app65}
\int_{-\infty}^{\infty} \int_0^{\infty} \big\{M_{\lambda}(x) - G_{\lambda}(x)\big\} \lambda^{-\frac{\sigma}{2} - 1}\ \dl\ \dx
	= 2 \,\gamma(1 + \sigma) \,\zeta(1 + \sigma).
\end{equation}
\end{lemma}

Theorems \ref{MinDis}, \ref{MajDis} and \ref{BADis} now apply. The values of the integrals in the following corollary are obtained from Lemmas \ref{applem5} and \ref{applem6}. 
\begin{corollary}\label{thm19}
Let $-1<\sigma$ with $\sigma \neq 0,2,4,...$ and let 
\begin{equation*}
 g_{\sigma}(x) = \gamma(-\sigma) |x|^{\sigma}.
\end{equation*}
\begin{enumerate} 
 \item[(i)] There exists a unique extremal minorant $l_{\sigma}(z)$ of exponential type $2\pi$ for $g_{\sigma}(x)$. The function $l_{\sigma}(x)$ interpolates the values of $g_{\sigma}(x)$ at $\Z + \h$ and satisfies
\begin{equation}\label{Sec11.1}
\int_{-\infty}^{\infty} \{g_{\sigma}(x) - l_{\sigma}(x)\} \, \dx  =  \bigl(2 - 2^{1-\sigma}\bigr) \gamma (1 +\sigma) \,\zeta(1+ \sigma).
\end{equation}
\item[(ii)] If $0<\sigma$, there exists a unique extremal majorant $m_{\sigma}(z)$ of exponential type $2\pi$ for $g_{\sigma}(x)$. The function $m_{\sigma}(x)$ interpolates the values of $g_{\sigma}(x)$ at $\Z$ and satisfies
\begin{equation}\label{Sec11.2}
\int_{-\infty}^{\infty} \{m_{\sigma}(x) - g_{\sigma}(x)\} \, \dx  =  2 \,\gamma (1+ \sigma) \,\zeta(1+ \sigma).
\end{equation}
\item[(iii)] There exists a unique best approximation $k_{\sigma}(z)$ of exponential type $\pi$ for $g_{\sigma}(x)$. The function $k_{\sigma}(x)$ interpolates the values of $g_{\sigma}(x)$ at $\Z + \h$, satisfying
\begin{equation*}
 \sgn(\cos \pi x)\{ g_{\sigma}(x) - k_{\sigma}(x)\} \geq 0
\end{equation*}
and
\begin{equation}\label{Sec11.3}
\int_{-\infty}^{\infty} \bigl|g_{\sigma}(x) - k_{\sigma}(x)\bigr| \, \dx  =  \left(\frac{4}{\pi}\right)^{\frac{3+\sigma}{2}} \Gamma\left(\frac{1 +\sigma}{2}\right) L(2 +\sigma, \chi) .
\end{equation}

\end{enumerate}

\end{corollary}

Corollary \ref{thm19} is the complete treatment for the functions $x \mapsto |x|^\sigma$ since for $\sigma \leq -1$ these functions are not integrable at the origin, and therefore no extremals exist, and for $\sigma = 2k, k \in \Z^{+}$, these functions are entire, and therefore the extremal problem is trivial. Previous results had been obtained in \cite{CV2} and \cite{CV3} for the functions $|x|^{\sigma}$, $-1<\sigma <1$ and in \cite{Lit} for the functions $|x|^{2k+1}$, $k \in \Z^{+}$.

Next we consider Hilbert-type inequalities.  It is well known that there is a simple relationship between the solution of the Beurling-Selberg extremal problem for a function $g:\R \to \R$ and the existence of optimal bounds for Hermitian forms involving the Fourier transform $\widehat{g}$.  Such bounds for
Hermitian forms are called Hilbert-type inequalities. Detailed proofs of these inequalities can be found for instance in \cite[Theorem 16]{V} or \cite[Theorem 7.1]{CV2}. In particular we report here the Hilbert-type inequalities that follow from Corollary \ref{thm19}. They involve the same kernel as the classical discrete Hardy-Littlewood-Sobolev inequality (see \cite[p. 288]{HPL}), and generalize the result contained in \cite[Corollary 7.2]{CV2}.

\begin{corollary}
 Let $\xi_1, \xi_2,...,\xi_N$ be real numbers such that $0 < \delta \leq |\xi_m-\xi_n|$ whenever $m\neq n$. Let $a_1, a_2,..., a_N$ be complex numbers. If $0 < \sigma < 1$ then 
\begin{equation*}
 -\frac{\bigl(2 - 2^{2-\sigma}\bigr)\zeta(\sigma)}{\delta^{\sigma}} \sum_{n=1}^N |a_n|^2 \leq \sum_{m=1}^N \sum_{\stackrel{n=1}{n\neq m}}^{N} \frac{a_m \overline{a}_n}{|\xi_m - \xi_n|^\sigma}\,,
\end{equation*}
if $\sigma = 1$ then,
\begin{equation*}
 -\frac{\log 4}{\delta} \sum_{n=1}^N |a_n|^2 \leq \sum_{m=1}^N \sum_{\stackrel{n=1}{n\neq m}}^{N} \frac{a_m \overline{a}_n}{|\xi_m - \xi_n|}\,,
\end{equation*}
and if $1 < \sigma$ then
\begin{equation*}
 -\frac{\bigl(2 - 2^{2-\sigma}\bigr)\zeta(\sigma)}{\delta^{\sigma}} \sum_{n=1}^N |a_n|^2 \leq \sum_{m=1}^N \sum_{\stackrel{n=1}{n\neq m}}^{N} \frac{a_m \overline{a}_n}{|\xi_m - \xi_n|^\sigma} \leq \frac{ 2\zeta(\sigma)}{\delta^{\sigma}} \sum_{n=1}^N |a_n|^2.
\end{equation*}
The constants appearing in these inequalities are the best possible.
\end{corollary}

\section{Further Examples}

We finish our list of applications with two additional examples contemplated by the distribution method.

\begin{corollary}
Let $\alpha\ge 0$ and consider
\begin{equation*}
x\mapsto\tau_{\alpha}(x) =  -\log(x^2+\alpha^2).
\end{equation*}
\begin{enumerate} 
 \item[(i)] There exists a unique extremal minorant $l_\alpha$ of exponential type $2\pi$ for $\tau_{\alpha}$. The function $l_{\sigma}$ interpolates the values of $\tau_{\alpha}$ at $\Z + \h$ and satisfies
\begin{equation*}
\int_{-\infty}^{\infty} \{\tau_{\alpha}(x) - l_{\alpha}(x)\} \, \dx  =  2\log\bigl(1 + e^{2\pi\alpha}\bigr).
\end{equation*}
\item[(ii)] If $0<\alpha$, there exists a unique extremal majorant $m_{\alpha}$ of exponential type $2\pi$ for $\tau_{\alpha}$. The function $m_{\alpha}$ interpolates the values of $\tau_{\alpha}$ at $\Z$ and satisfies
\begin{equation*}
\int_{-\infty}^{\infty} \{m_{\alpha}(x) - \tau_{\alpha}(x)\} \, \dx  =  2\log\bigl(1 - e^{2\pi\alpha}\bigr).
\end{equation*}
\item[(iii)] There exists a unique best approximation $k_{\alpha}$ of exponential type $\pi$ for $\tau_{\alpha}$. The function $k_{\alpha}$ interpolates the values of $\tau_{\alpha}$ at $\Z + \h$, satisfying
\begin{equation*}
\sgn(\cos \pi x)\{ \tau_{\alpha}(x) - k_{\alpha}(x)\} \geq 0
\end{equation*}
and
\begin{equation*}
\int_{-\infty}^{\infty} |\tau_{\alpha}(x) - k_{\alpha}(x)|\, \dx  =  \int_0^{\infty} \left\{\frac{1}{\pi} \sum_{n=-\infty}^{\infty} \frac{(-1)^n}{n+\hh} \,\tG_{\lambda}(n + \hh)\right\} \frac{e^{-\pi\lambda \alpha^2}}{\lambda}\, \dl.
\end{equation*}

\end{enumerate}
\end{corollary}
 
\begin{proof} For $0 \leq \alpha$ we have the following identity
\begin{equation}\label{idlog}
-\log(x^2 + \alpha^2) = \int_{0}^{\infty} \frac{\bigl\{e^{-\pi\lambda (x^2 + \alpha^2)} - e^{-\pi \lambda}\bigr\}}{\lambda} \, \dl\,.
\end{equation}
Let $\varphi$ be a Schwartz function supported in $[-\delta,\delta]^c$. An application of Fubini's theorem gives us
\begin{align}\label{Sec12.1}
\begin{split}
\int_{-\infty}^{\infty} -\log(x^2 + &\alpha^2)\,\widehat{\varphi}(x)\,\dx \\
& = \int_{-\infty}^{\infty} \left\{ \int_{0}^{\infty} \frac{\bigl\{e^{-\pi\lambda (x^2 + \alpha^2)} - e^{-\pi \lambda}\bigr\}}{\lambda}  \dl\right\} \,\widehat{\varphi}(x)\, \dx\\
& = \int_{0}^{\infty} \int_{-\infty}^{\infty} \frac{\bigl\{e^{-\pi\lambda (x^2 + \alpha^2)} - e^{-\pi \lambda}\bigr\}}{\lambda}\,\widehat{\varphi}(x)\,\dx\,\dl\\
& = \int_{0}^{\infty} \left\{ \int_{-\infty}^{\infty} \widehat{G}_{\lambda}(t)\,\varphi(t) \,\dt\right\} \frac{e^{-\pi\lambda \alpha^2}}{\lambda}\, \dl\\
& = \int_{-\infty}^{\infty} \left\{\int_{0}^{\infty} \widehat{G}_{\lambda}(t)\,\frac{e^{-\pi\lambda \alpha^2}}{\lambda}\, \dl\right\}\, \varphi(t) \,\dt.
\end{split}
\end{align}
Equation (\ref{Sec12.1}) provides the Fourier transform of $-\log(x^2 + \alpha^2)$ outside a compact interval $[-\delta, \delta]$. We can therefore apply the distribution method (Theorems \ref{MinDis}, \ref{MajDis} and \ref{BADis}) with the Gaussian family and measure $\nu$ on $I = (0,\infty)$ given by
\begin{equation*}
 \dnu = \frac{e^{-\pi\lambda \alpha^2}}{\lambda}\, \dl.
\end{equation*}
According to the asymptotics (\ref{Sec10.4}), (\ref{Sec10.5}) and (\ref{Sec10.6}), if $\alpha >0$ we can treat the three approximation problems, and if $\alpha =0$ we can only treat the minorant and the best approximation problem (which is in agreement with the fact that $-\log |x|$ is unbounded by above). The special case of $-\log|x|$ (when $\alpha =0$) was achieved in the papers \cite{CV2} and \cite{CV3}.
\end{proof}

\begin{corollary} Let $n\in\N$ and define $h_n$ by
\begin{equation*}
 h_{n}(x) = (-1)^{n+1} x^{2n}\log(x^2).
\end{equation*}
\begin{enumerate} 
 \item[(i)] There exists a unique extremal minorant $l_n$ of exponential type $2\pi$ for $h_n$. The function $l_n$ interpolates the values of $h_n$ at $\Z + \h$ and satisfies
\begin{equation*}
\int_{-\infty}^{\infty} \{h_n(x) - l_n(x)\} \, \dx  =  \bigl(2 - 2^{1-2n}\bigr) (2n)!\, (2\pi)^{-2n} \zeta(2n+1).
\end{equation*}
\item[(ii)] If $n>0$, there exists a unique extremal majorant $m_n$ of exponential type $2\pi$ for $h_n$. The function $m_n$ interpolates the values of $h_n$ at $\Z$ and satisfies
\begin{equation*}
\int_{-\infty}^{\infty} \{m_n(x) - h_n(x)\} \, \dx  =2(2n)!\,(2\pi)^{-2n} \zeta(2n+1).
\end{equation*}
\item[(iii)] There exists a unique best approximation $k_n$ of exponential type $\pi$ for $h_n$. The function $k_n$ interpolates the values of $h_n$ at $\Z + \h$, satisfying
\begin{equation*}
 \sgn(\cos \pi x)\{ h_n(x) - k_n(x)\} \geq 0
\end{equation*}
and
\begin{equation*}
\int_{-\infty}^{\infty} \bigl|h_n(x) - k_n(x)\bigr| \, \dx  =  \frac{2}{\pi} (2n)!\,(2\pi)^{-2n} L(\chi,2+2n)
\end{equation*}
\end{enumerate}
\end{corollary}
\begin{proof}
Let $\varphi$ be a Schwartz function supported in $[-\delta,\delta]^c$. We make use of identity (\ref{idlog}) (with $\alpha =0$) and repeated applications of Fubini's theorem to obtain 
\begin{align}\label{Sec12.2}
\begin{split}
\int_{-\infty}^{\infty} h_n(x) \,\widehat{\varphi}(x)\,\dx 
& = \int_{-\infty}^{\infty} (-1)^n \,x^{2n}\,\left\{ \int_{0}^{\infty} \frac{\bigl\{e^{-\pi\lambda x^2 } - e^{-\pi \lambda}\bigr\}}{\lambda}  \dl\right\} \,\widehat{\varphi}(x)\, \dx\\
& = \int_{0}^{\infty} \int_{-\infty}^{\infty} (-1)^n \,x^{2n}\,\frac{\bigl\{e^{-\pi\lambda x^2} - e^{-\pi \lambda}\bigr\}}{\lambda}\,\widehat{\varphi}(x)\,\dx\,\dl\\
& = \frac{1}{(2\pi)^{2n}}\int_{0}^{\infty} \left\{ \int_{-\infty}^{\infty} \widehat{G}_{\lambda}^{(2n)}(t)\,\varphi(t) \,\dt\right\} \frac{1}{\lambda}\, \dl\\
& = \frac{1}{(2\pi)^{2n}}\int_{-\infty}^{\infty} \left\{\int_{0}^{\infty} \frac{\widehat{G}_{\lambda}^{(2n)}(t)}{\lambda}\, \dl\right\}\, \varphi(t) \,\dt\\
& = \frac{1}{(2\pi)^{2n}}\int_{-\infty}^{\infty} \left(\frac{\text{\rm d}}{\dt}\right)^{2n}\left\{\int_{0}^{\infty} \frac{\widehat{G}_{\lambda}(t)}{\lambda}\, \dl\right\}\, \varphi(t) \,\dt\,,
\end{split}
\end{align}
and using (\ref{sigmameasure}), the last integral translates to 
\begin{align}
 \begin{split}\label{Sec12.3}
 &= \frac{\gamma(1)}{(2\pi)^{2n}}\int_{-\infty}^{\infty} \left(\frac{\text{\rm d}}{\dt}\right)^{2n}\,|t|^{-1}\, \varphi(t) \,\dt\\
& = \frac{ (2n)!}{(2\pi)^{2n}} \int_{-\infty}^{\infty} |t|^{-2n-1}\, \varphi(t) \,\dt.
 \end{split}
\end{align}
Hence, the Fourier transform of $h_n$ in the distribution sense, outside the compact interval $[-\delta, \delta]$, is given by the function
\begin{equation*}
 t \mapsto \frac{(2n)!}{(2\pi)^{2n}}|t|^{-2n-1}\, =  \frac{(2n)!}{(2\pi)^{2n}} \gamma(2n+1)^{-1} \int_0^{\infty} \widehat{G}_{\lambda}(t) \,\text{\rm d}\nu_{2n}(\lambda),
\end{equation*}
where the last identity follows from (\ref{sigmameasure}) with $\sigma=2n$. An application of Theorems \ref{MinDis}, \ref{MajDis} and \ref{BADis} with measure
\[
\dnu = \frac{(2n)!}{(2\pi)^{2n}} \gamma(2n+1)^{-1} \lambda^{-n-1} \dl
\]
together with the formulas in Lemmas \ref{applem5} and \ref{applem6} give the desired result.

\end{proof}

\section*{Acknowledgments}
This material is based upon work supported by the National Science Foundation under agreements No. DMS-0635607 (E. Carneiro) and DMS-0603282 (J. D. Vaaler). Any opinions, findings and conclusions or recommendations expressed in this material are those of the authors and do not necessarily reflect the views of the National Science Foundation. E. Carneiro would also like to acknowledge support from the Capes/Fulbright grant BEX 1710-04-4 and the Homer Lindsey Bruce Fellowship from the University of Texas. The authors are thankful to K. Soundararajan and E. Bombieri for enlightening comments during the preparation of this work.

\end{document}